%% file: EvenSpinKod.tex
\documentclass[reqno]{amsart}

\usepackage[english]{babel}
\usepackage{csquotes}
\usepackage[style=alphabetic,
            isbn=false,
            doi=false,
            url=false,
            backend=bibtex,
            maxnames=5, 
            maxalphanames=5,
             firstinits=true
           ]{biblatex}
\bibliography{sources}
\usepackage{amsmath,amsfonts,amsthm,mathtools, amssymb}
\usepackage{xcolor}
\usepackage{tikz-cd}
\usepackage{hyperref}
\usepackage{aliascnt}
\usepackage{xparse}
\usepackage{tensor}
\usepackage{caption}
\usepackage{ mathrsfs }
\usepackage{tabto}
\usepackage{imakeidx}
\usepackage[normalem]{ulem}
\usepackage{comment}
\allowdisplaybreaks

 \RequirePackage{filecontents}

\indexsetup{level=\subsection*,toclevel=\paragraph,headers={ \MakeUppercase{Euler sequence on strata}}{\MakeUppercase{Euler sequence on strata}}, noclearpage,firstpagestyle=headings}
\makeindex[name=graph,title=Graphs and levels,columns=1,columnseprule,options=-s Idx.ist]

\usepackage{tikz}
\usetikzlibrary{matrix,arrows,calc}
\usetikzlibrary{positioning}
\usetikzlibrary{decorations.pathreplacing,decorations.markings,decorations.pathmorphing}
\usetikzlibrary{positioning,arrows,patterns}
\usetikzlibrary{cd}
\usetikzlibrary{intersections}

\tikzset{
	every loop/.style={very thick},
	comp/.style={circle,fill,black,inner sep=0pt,minimum size=5pt},
	order bottom left/.style={pos=.05,left,font=\tiny},
	order top left/.style={pos=.9,left,font=\tiny},
	order bottom right/.style={pos=.05,right,font=\tiny},
	order top right/.style={pos=.9,right,font=\tiny},
	order node dis/.style={text width=.75cm},
	circled number/.style={circle, draw, inner sep=0pt, minimum size=12pt},
	below left with distance/.style={below left,text height=10pt},
    below right with distance/.style={below right,text height=10pt}
	}

\makeatletter
\ifcsname phantomsection\endcsname
    \newcommand*{\@gobblenexttocentry}[9]{}
\else
    \newcommand*{\@gobblenexttocentry}[4]{}
\fi
\newcommand*{\addsubsection}{%
    \addtocontents{toc}{\protect\@gobblenexttocentry}%
    \subsection*}
\makeatother

\DeclareFieldFormat[article]{title}{{\it #1}}
\DeclareFieldFormat{journaltitle}{{\rm #1}}
\renewbibmacro{in:}{%
  \ifentrytype{article}{}{\printtext{\bibstring{in}\intitlepunct}}}
\DeclareFieldFormat[incollection]{title}{{\it #1}}
\DeclareFieldFormat{journaltitle}{{\rm #1}}

\begin{document}

\def\subsectionautorefname{Section}
\def\subsubsectionautorefname{Section}
\def\sectionautorefname{Section}
\def\equationautorefname~#1\null{(#1)\null}

\include{macros}


\title[Kodaira dimension of even spin strata]
      {The Kodaira dimension of even spin strata of Abelian differentials}

\author{Andrei Bud}
\address{Institut f\"ur Mathematik, Goethe-Universit\"at Frankfurt,
	Robert-Mayer-Str. 6-8,
	60325 Frankfurt am Main, Germany}
\email{andreibud95@protonmail.com}
\thanks{Research of A.B.\ is supported by the Collaborative Research Centre
	TRR 326 ``Geometry and Arithmetic of Uniformized Structures.''}

\author{Dawei Chen}
      \thanks{Research of D.C. is supported in part by National Science Foundation Grant DMS-2301030, 
        Simons Travel Support for Mathematicians, and a Simons Fellowship.}
\address{Department of Mathematics, Boston College, Chestnut Hill, MA 02467, USA}
\email{dawei.chen@bc.edu}

\author{Martin M\"oller}
\address{Institut f\"ur Mathematik, Goethe-Universit\"at Frankfurt,
Robert-Mayer-Str. 6-8,
60325 Frankfurt am Main, Germany}
\email{moeller@math.uni-frankfurt.de}
\thanks{Research of M.M.\ is supported
by the DFG-project MO 1884/2-1, 
and the Collaborative Research Centre
TRR 326 ``Geometry and Arithmetic of Uniformized Structures.''}

\begin{abstract}
The even spin components of the strata of Abelian differentials are difficult to handle from a birational geometry perspective due to the fact that their spin line bundles have more sections than expected. Nevertheless, we prove that for large genus, the minimal even spin components are of general type. This result complements the previous work by the second and third authors, together with Costantini, on the Kodaira dimension of general strata and the minimal odd spin components of Abelian differentials. Our main technical tool is the computation and estimation of a series of effective divisor classes on the even spin components. 
\end{abstract}
\maketitle
\tableofcontents

\input{sec_intro}

\input{sec_MS}

\input{sec_pullback}
\input{sec_BN}
\input{sec_genWP}

\input{sec_gentype}
\printbibliography

\end{document}

%% file: macros.tex
\newcommand{\mynewtheorem}[4]{
  \if\relax\detokenize{#3}\relax 
    \if\relax\detokenize{#4}\relax 
      \newtheorem{#1}{#2}
    \else
      \newtheorem{#1}{#2}[#4]
    \fi
  \else
    \newaliascnt{#1}{#3}
    \newtheorem{#1}[#1]{#2}
    \aliascntresetthe{#1}
  \fi
  \expandafter\def\csname #1autorefname\endcsname{#2}
}

\mynewtheorem{theorem}{Theorem}{}{section}
\mynewtheorem{lemma}{Lemma}{theorem}{}
\mynewtheorem{rem}{Remark}{lemma}{}
\mynewtheorem{prop}{Proposition}{lemma}{}
\mynewtheorem{cor}{Corollary}{lemma}{}
\mynewtheorem{add}{Addendum}{lemma}{}
\mynewtheorem{definition}{Definition}{lemma}{}
\mynewtheorem{question}{Question}{lemma}{}
\mynewtheorem{assumption}{Assumption}{lemma}{}
\mynewtheorem{example}{Example}{lemma}{}


\def\defbb#1{\expandafter\def\csname b#1\endcsname{\mathbb{#1}}}
\def\defcal#1{\expandafter\def\csname c#1\endcsname{\mathcal{#1}}}
\def\deffrak#1{\expandafter\def\csname frak#1\endcsname{\mathfrak{#1}}}
\def\defop#1{\expandafter\def\csname#1\endcsname{\operatorname{#1}}}
\def\defbf#1{\expandafter\def\csname b#1\endcsname{\mathbf{#1}}}

\makeatletter
\def\defcals#1{\@defcals#1\@nil}
\def\@defcals#1{\ifx#1\@nil\else\defcal{#1}\expandafter\@defcals\fi}
\def\deffraks#1{\@deffraks#1\@nil}
\def\@deffraks#1{\ifx#1\@nil\else\deffrak{#1}\expandafter\@deffraks\fi}
\def\defbbs#1{\@defbbs#1\@nil}
\def\@defbbs#1{\ifx#1\@nil\else\defbb{#1}\expandafter\@defbbs\fi}
\def\defbfs#1{\@defbfs#1\@nil}
\def\@defbfs#1{\ifx#1\@nil\else\defbf{#1}\expandafter\@defbfs\fi}
\def\defops#1{\@defops#1,\@nil}
\def\@defops#1,#2\@nil{\if\relax#1\relax\else\defop{#1}\fi\if\relax#2\relax\else\expandafter\@defops#2\@nil\fi}
\makeatother

\defbbs{ZHQCNPALRVW}
\defcals{ABOPQMNXYLTRAEHZKCFI}
\deffraks{apijklgmnopqueRC}
\defops{IVC, PGL,SL,mod,Spec,Re,Gal,Tr,End,GL,Hom,PSL,H,div,Aut,rk,Mod,R,T,Tr,Mat,Vol,MV,Res,Hur, vol,Z,diag,Hyp,hyp,hl,ord,Im,ev,U,dev,c,CH,fin,pr,Pic,lcm,ch,td,LG,id,Sym,Aut,Log,tw,irr,discrep,BN,NF,NC,age,hor,lev,ram,NH,av,app,mid}
\defbfs{cuvzwp} 

\def\ep{\varepsilon}
\def\ve{\varepsilon}
\def\abs#1{\lvert#1\rvert}
\def\dd{\mathrm{d}}
\def\WP{\mathrm{WP}}
\def\inj{\hookrightarrow}
\def\eq{=}

\def\i{\mathrm{i}}
\def\e{\mathrm{e}}
\def\st{\mathrm{st}}
\def\ct{\mathrm{ct}}
\def\rel{\mathrm{rel}}
\def\odd{\mathrm{odd}}
\def\even{\mathrm{even}}

\def\uC{\underline{\bC}}
\def\ol{\overline}
  
\def\Vrel{\bV^{\mathrm{rel}}}
\def\Wrel{\bW^{\mathrm{rel}}}
\def\twolev{\mathrm{LG_1(B)}}

\def\be{\begin{equation}}   \def\ee{\end{equation}}     \def\bes{\begin{equation*}}    \def\ees{\end{equation*}}
\def\ba{\be\begin{aligned}} \def\ea{\end{aligned}\ee}   \def\bas{\bes\begin{aligned}}  \def\eas{\end{aligned}\ees}
\def\={\;=\;}  \def\+{\,+\,} \def\m{\,-\,}

\newcommand*{\proj}{\mathbb{P}}
\newcommand{\IVCst}[1][\mu]{{\mathcal{IVC}}({#1})}
\newcommand{\barmoduli}[1][g]{{\overline{\mathcal M}}_{#1}}
\newcommand{\moduli}[1][g]{{\mathcal M}_{#1}}
\newcommand{\omoduli}[1][g]{{\Omega\mathcal M}_{#1}}
\newcommand{\modulin}[1][g,n]{{\mathcal M}_{#1}}
\newcommand{\omodulin}[1][g,n]{{\Omega\mathcal M}_{#1}}
\newcommand{\zomoduli}[1][]{{\mathcal H}_{#1}}
\newcommand{\barzomoduli}[1][]{{\overline{\mathcal H}_{#1}}}
\newcommand{\pomoduli}[1][g]{{\proj\Omega\mathcal M}_{#1}}
\newcommand{\pomodulin}[1][g,n]{{\proj\Omega\mathcal M}_{#1}}
\newcommand{\pobarmoduli}[1][g]{{\proj\Omega\overline{\mathcal M}}_{#1}}
\newcommand{\pobarmodulin}[1][g,n]{{\proj\Omega\overline{\mathcal M}}_{#1}}
\newcommand{\potmoduli}[1][g]{\proj\Omega\tilde{\mathcal{M}}_{#1}}
\newcommand{\obarmoduli}[1][g]{{\Omega\overline{\mathcal M}}_{#1}}
\newcommand{\obarmodulio}[1][g]{{\Omega\overline{\mathcal M}}_{#1}^{0}}
\newcommand{\otmoduli}[1][g]{\Omega\tilde{\mathcal{M}}_{#1}}
\newcommand{\pom}[1][g]{\proj\Omega{\mathcal M}_{#1}}
\newcommand{\pobarm}[1][g]{\proj\Omega\overline{\mathcal M}_{#1}}
\newcommand{\pobarmn}[1][g,n]{\proj\Omega\overline{\mathcal M}_{#1}}
\newcommand{\princbound}{\partial\mathcal{H}}
\newcommand{\omoduliinc}[2][g,n]{{\Omega\mathcal M}_{#1}^{{\rm inc}}(#2)}
\newcommand{\obarmoduliinc}[2][g,n]{{\Omega\overline{\mathcal M}}_{#1}^{{\rm inc}}(#2)}
\newcommand{\pobarmoduliinc}[2][g,n]{{\proj\Omega\overline{\mathcal M}}_{#1}^{{\rm inc}}(#2)}
\newcommand{\otildemoduliinc}[2][g,n]{{\Omega\widetilde{\mathcal M}}_{#1}^{{\rm inc}}(#2)}
\newcommand{\potildemoduliinc}[2][g,n]{{\proj\Omega\widetilde{\mathcal M}}_{#1}^{{\rm inc}}(#2)}
\newcommand{\omoduliincp}[2][g,\lbrace n \rbrace]{{\Omega\mathcal M}_{#1}^{{\rm inc}}(#2)}
\newcommand{\obarmoduliincp}[2][g,\lbrace n \rbrace]{{\Omega\overline{\mathcal M}}_{#1}^{{\rm inc}}(#2)}
\newcommand{\obarmodulin}[1][g,n]{{\Omega\overline{\mathcal M}}_{#1}}
\newcommand{\LTH}[1][g,n]{{K \overline{\mathcal M}}_{#1}}
\newcommand{\PLS}[1][g,n]{{\bP\Xi \mathcal M}_{#1}}

\DeclareDocumentCommand{\LMS}{ O{\mu} O{g,n} O{}}{\Xi\overline{\mathcal{M}}^{#3}_{#2}(#1)}
\DeclareDocumentCommand{\Romod}{ O{\mu} O{g,n} O{}}{\Omega\mathcal{M}^{#3}_{#2}(#1)}

\newcommand*{\Tw}[1][\Lambda]{\mathrm{Tw}_{#1}}  
\newcommand*{\sTw}[1][\Lambda]{\mathrm{Tw}_{#1}^s}  

\def\OO{\mathcal{O}}
\def\mm{\overline{\mathcal{M}}}

\newcommand{\HH}{{\mathbb H}}
\newcommand{\MM}{{\mathbb M}}
\newcommand{\bbC}{{\mathbb C}}

\newcommand{\bfa}{{\bf a}}
\newcommand{\bfb}{{\bf b}}
\newcommand{\bfc}{{\bf c}}
\newcommand{\bfd}{{\bf d}}
\newcommand{\bfe}{{\bf e}}
\newcommand{\bff}{{\bf f}}
\newcommand{\bfg}{{\bf g}}
\newcommand{\bfh}{{\bf h}}
\newcommand{\bfm}{{\bf m}}
\newcommand{\bfn}{{\bf n}}
\newcommand{\bfp}{{\bf p}}
\newcommand{\bfq}{{\bf q}}
\newcommand{\bft}{{\bf t}}
\newcommand{\bfP}{{\bf P}}
\newcommand{\bfR}{{\bf R}}
\newcommand{\bfU}{{\bf U}}
\newcommand{\bfu}{{\bf u}}
\newcommand{\bfx}{{\bf x}}
\newcommand{\bfz}{{\bf z}}

\newcommand{\bfl}{{\boldsymbol{\ell}}}
\newcommand{\bfmu}{{\boldsymbol{\mu}}}
\newcommand{\bfeta}{{\boldsymbol{\eta}}}
\newcommand{\bftau}{{\boldsymbol{\tau}}}
\newcommand{\bfomega}{{\boldsymbol{\omega}}}
\newcommand{\bfsigma}{{\boldsymbol{\sigma}}}

\newcommand{\wh}{\widehat}
\newcommand{\wt}{\widetilde}

\newcommand{\ps}{\mathrm{ps}}  
\newcommand{\CPT}{\mathrm{CPT}}  
\newcommand{\NCT}{\mathrm{NCT}}  
\newcommand{\RBD}{\mathrm{RBD}}
\newcommand{\ETD}{\mathrm{ETD}}
\newcommand{\OCT}{\mathrm{OCT}}
\newcommand{\SRT}{\mathrm{SRT}}
\newcommand{\RBT}{\mathrm{RBT}}

\newcommand{\tdpm}[1][{\Gamma}]{\mathfrak{W}_{\operatorname{pm}}(#1)}
\newcommand{\tdps}[1][{\Gamma}]{\mathfrak{W}_{\operatorname{ps}}(#1)}

\newlength{\halfbls}\setlength{\halfbls}{.8\baselineskip}

\newcommand*{\Teichmuller}{Teich\-m\"uller\xspace}

\DeclareDocumentCommand{\MSfun}{ O{\mu} }{\mathbf{MS}({#1})}
\DeclareDocumentCommand{\MSgrp}{ O{\mu} }{\mathcal{MS}({#1})}
\DeclareDocumentCommand{\MScoarse}{ O{\mu} }{\mathrm{MS}({#1})}
\DeclareDocumentCommand{\tMScoarse}{ O{\mu} }{\widetilde{\mathbb{P}\mathrm{MS}}({#1})}

\newcommand{\kmin}{\kappa_{(2g-2)}}
\newcommand{\ktop}{\kappa_{\mu_\Gamma^{\top}}}
\newcommand{\kbot}{\kappa_{\mu_\Gamma^{\bot}}}

\def\changed#1{\textcolor{red}{#1}}

%% file: sec_intro.tex
\section{Introduction}

Let $\mu = (m_1,\ldots, m_n)$ be a tuple of integers where $\sum_{i=1}^n m_i = 2g-2$. The moduli space $\bP\omoduli[g,n](\mu)$ parameterizes smooth and connected complex algebraic curves $(X, z_1, \ldots, z_n)$ of genus $g$ with $n$ marked points such that $\sum_{i=1}^n m_iz_i$ is a canonical divisor of type $\mu$ in $X$. As $\mu$ varies, these spaces $\bP\omoduli[g,n](\mu)$ stratify the (projectivized) Hodge bundle of Abelian differentials according to the types of the zero and pole orders of the underlying differentials. Therefore, $\bP\omoduli[g,n](\mu)$ is also called the (projectivized) stratum of Abelian differentials of type $\mu$.

The goal of this paper is to improve the understanding of the birational geometry of strata of Abelian differentials in the series of components that have been left aside in previous work: the even spin case. It is well-known that these strata $\bP\omoduli[g,n](\mu)$
can have up to three connected components in the holomorphic case $m_i \geq 0$ for all~$i$, according to the classification by Kontsevich--Zorich \cite{kozo1}. In the case when the strata are connected, they behave as expected for many moduli spaces: low-genus examples are uniruled \cite{Bar18, BudGonality}, while for large genus, many patterns of zeros, including cases with both few and many zeros, have their Kodaira dimension shown to be maximal in \cite{Kodstrata}, i.e., being of general type. Here, we complement the understanding in the case of partitions~$\mu$ where the strata are disconnected.
\par
If $\mu = (g-1,g-1)$ or $\mu = (2g-2)$, then the strata have a hyperelliptic component $\bP\omoduli[g,n](\mu)^{\hyp}$, and these components are always rational by the explicit parametrization of hyperelliptic curves. If all entries of~$\mu$ are even and $g\geq 4$, then the strata can have odd spin and even spin components whose generic elements are not hyperelliptic. 
For odd spin and the many types alluded to above, it was also shown in \cite{Kodstrata} that these components are of general type in large genus.
\par
The most interesting remaining case is the series of even spin components of the minimal strata, i.e., $\mu = (2g-2)^{\even}$. This is because the even spin component is entirely contained in the unique generalized Weierstrass divisor (see Section~\ref{sec:genWP}) which was crucially used for all other strata dealt with in \cite{Kodstrata}. Our main result below is to overcome this difficulty, to apply the usual strategy of ``canonical is ample plus effective," and to show that even and odd spin components behave similarly, contrasting with the hyperelliptic case.
\par
\begin{theorem} \label{intro:minimal}
The even spin components of the coarse moduli spaces of Abelian differentials
with a unique zero, $\bP\Omega \mathrm{M}_{g,1}(2g-2)^{\even}$, are of general type for
$g = 31$ and $g \geq 33$.
\end{theorem}
\par
\medskip
\paragraph{\textbf{Methods of the proof and possible improvements}} Our strategy for proving Theorem~\ref{intro:minimal} is to compute the class of the generalized Weierstrass divisor on the multi-scale compactification of the stratum $ \bP\LMS[\mu][g+1,n]$ and pull it back via the clutching map
\bes
{\zeta_E}\colon \bP\LMS[2g-2][g,1]^{\even}  \rightarrow \bP\LMS[\mu][g+1,n]\,,
\ees
which is defined in Section \ref{sec:CTclutching}. The intuition for this method comes from the case of moduli spaces of curves, where pulling back the class of the Weierstrass divisor via the map 
$\pi_E \colon \ol{\mathcal{M}}_{g,1} \rightarrow \ol{\mathcal{M}}_{g+1,1}$, yields the class of the Weierstrass divisor in genus one lower. The expectation that a similar phenomenon occurs for generalized Weierstrass divisors is confirmed in Theorem \ref{theorem:class_twisted_Weierstrass}. 
\par
The technical challenge in pulling back generalized Weierstrass divisors via the map~$\zeta_E$ is to ensure that the effective divisor in $\bP\LMS[\mu][g+1,n]$ that we pull back is not supported on the boundary divisor $D_{\Gamma_1}$, which contains $\textrm{Im}(\zeta_E)$.  The generalized Weierstrass divisor is defined as the degeneracy locus of a morphism of vector bundles of rank $g-1$: 
\[\mathcal{H}/\OO(-1) \rightarrow \mathcal{F}_{\alpha}\,.\] 
To compute the vanishing order of this morphism along $D_{\Gamma_1}$, we use a generalization of the theory of limit linear series and the methods outlined in \cite{Cukierman} and \cite[Section~2]{limitlinearbasic}, namely, linear series with multi-vanishing orders (at a chain of divisors) as discussed in \cite{Ossermancompacttype}. The vanishing order computation is presented in Proposition~\ref{prop: exact multiplicity}, which relies crucially on the multi-vanishing order computations in Proposition~\ref{prop:even-general-vanishing}.
\par \par
The genus bound in the conclusion above is not the optimal result, but rather chosen to provide a concise proof in the case of large genus (see Section~\ref{sec:gentype}).  
Using computer-assisted checks for all coarse types of graphs, as in the complementary computer files to \cite{Kodstrata}, one could likely improve these bounds, presumably (by checking only some critical cases) to $g \geq 19$. We are aware of several minor improvements: (i) The divisor class $[\NF_{(2g-2)}]$ (see \cite[Lemma~6.9]{Kodstrata}) is marginally better in low even genus than the class $[\Hur_\mu]$, which we use exclusively here. (ii) All the divisor classes we use may have contributions entirely contained in the boundary, though we apply twisted versions of the generalized Weierstrass divisor to exclude them for some boundary strata. We are also aware that the pullback of the Brill--Noether divisor contains, for example, some ``banana" boundary divisors, which we could remove.
\par
However, none of the improvements known to us will bring us close to $g=12$, where the currently known strategies for proving uniruledness end. Finding the precise tipping point of the Kodaira dimension likely requires better effective divisors or improved control of the boundary contribution.
\medskip
\paragraph{\textbf{Even spin components of non-minimal strata}} The case of even spin and non-minimal strata can be addressed in the same way as in \cite{Kodstrata}, though this may introduce some rounding error when using the generalized Weierstrass divisor. We elaborate on this comment in Section~\ref{subsec:Weierstrass-nonminimal}. We show that the obvious ``generalized Weierstrass'' type condition indeed yields a divisor in Proposition~\ref{prop:properdivisor}, but we do not pursue the remaining steps, including the class computation or the numerical optimization. (The divisor class will be similar to that in \cite[Section~7]{Kodstrata}, although the twisting to minimize boundary contributions will need minor adjustments.)
\par
An interesting side remark is that generalized Weierstrass divisors in the even spin strata are not irreducible, not even in the interior $\bP\omoduli[g,n](\mu)^\even$ of the stratum (see Proposition~\ref{prop:genWisreducible}).

%% file: sec_MS.tex
\section{The moduli space of multi-scale differentials}
\label{sec:MS}

In this section, we recall basic facts about the compactification $\bP\LMS$ of
strata by multi-scale differentials, its singularities, its canonical class, and the classes of some pullback divisors. We assume familiarity with the
terminology of this compactification (see \cite{LMS} or as summarized
in \cite{CMZEuler,Kodstrata}). We denote by $\varphi\colon \bP\LMS \to \bP\MScoarse$
the map from the smooth Deligne--Mumford stack to its underlying coarse moduli space.
\par
\medskip
\paragraph{\textbf{Boundary divisors}}
Boundary strata of the compactification $\bP\LMS$ are indexed by level graphs (see \cite{LMS, CMZEuler}). For a level graph~$\Gamma$, we denote by $D_\Gamma$ the closure of the boundary stratum of multi-scale differentials with level graph~$\Gamma$. In general, the boundary strata~$D_\Gamma$ are not connected. Level graphs here are always enhanced by an integer $p_e \geq 0$ (the number of prongs) for each edge~$e$. The divisorial boundary strata (for $\mu$ of holomorphic type) consist of the horizontal boundary divisor~$D_h$, whose level graph has one vertex and one horizontal edge, and $D_\Gamma$, where $\Gamma$ is a two-level graph with no horizontal edges. We denote the set of two-level graphs without horizontal edges by $\LG_1$, where the index means ``one level below zero.'' For $\Gamma \in \LG_1$, we let $\ell_\Gamma = \lcm(p_e: e \in E(\Gamma))$.
\par
Each of the levels of the multi-scale differentials parametrized by~$D_\Gamma$
is commensurable (see \cite[Section~4.2]{CMZEuler} for details) with a
product of strata, possibly for disconnected curves and possibly constrained
by residue conditions. If~$\Gamma$ has two levels, we denote by~$N_\Gamma^\top$
and~$N^\bot_\Gamma$, respectively, the dimension of the unprojectivized top and bottom levels. They are related by the equation 
\be
N_\Gamma^\top + N_\Gamma^\bot \= N \coloneqq \dim\bP\LMS + 1
\ee
to the unprojectivized strata dimension. 
\par
\medskip
\paragraph{\textbf{Singularities}} The moduli space of multi-scale differentials $\bP\LMS$ is a smooth Deligne--Mumford stack (see~\cite[Theorem 1.3]{LMS}).  The underlying coarse moduli space $\bP\MScoarse$
has canonical singularities in its interior, as shown in \cite[Theorem~1.2]{Kodstrata}, 
but non-canonical singularities can appear at its boundary. Nevertheless, using a non-canonical compensation
divisor, the usual strategy to certify general type still works.
More precisely, consider the following types of graphs.
If an edge corresponds to a separating node, we say that it is of {\em compact
type (CPT)}. Otherwise, it is of {\em non-compact type (NCT)}.
If the lower part of the graph, separated by a CPT edge, consists of a single
rational vertex, the edge type is called a {\em rational bottom tail (RBT)}. 
A {\em (vertical) dumbbell (VDB)} graph is defined as a compact-type graph with a unique (separating) vertical edge. If the graph contains a unique vertical edge (i.e., a VDB graph), and if one end of the edge is 
genus one, we call it an {\em elliptic dumbbell  (EDB)}. An edge of
compact type that is neither RBT nor EDB is called {\em other compact
type (OCT)} (see \cite[Section~5]{Kodstrata}). Let
\ba \label{eq:non-canonical-term}
D_{\mathrm{NC}} & \,\coloneqq \,\sum_{\Gamma \in \LG_1} b_{\NC}^\Gamma [D_\Gamma] \,\coloneqq\, \sum_{\Gamma \in \LG_1}
(\ell_\Gamma R_{\NC}^\Gamma -1) [D_\Gamma]\, \quad \text{where}  \\
R_{\NC}^\Gamma &=\, \sum_{\rm NCT} \frac{1}{2}\frac{1}{p_e} + 
\sum_{\rm RBT} \frac{1}{p_{e}}  + \sum_{\rm OCT}
\frac{2}{ p_{e}} +\sum_{\rm EDB} \frac{4}{p_{e}} \,.
\ea
In the above, each sum runs over the edges of the corresponding type. The edge
type EDB is exclusive; that is, if it appears in a two-level graph, then the graph has a unique edge, and all other edge types do not appear.
\par
We can now state the criterion for the coarse moduli spaces of holomorphic differentials to be of general type.
\par
\begin{prop}[{\cite[Proposition~1.3]{Kodstrata}}] \label{prop:GenTypeCrit}
The effective divisor class $D_{\NC}$ has the property that pluri-canonical forms
associated with the perturbed canonical class  $K_{\bP\MScoarse} - D_{\NC}$ in the
smooth locus of $\bP\MScoarse$ extend to a desingularization.  
\par
In particular, if one can write
\be \label{eq:KDAE}
K_{\bP\MScoarse} - D_{\NC} \= A + E
\ee
with~$A$ being an ample divisor class and $E$ an effective divisor class,
then~$\bP\MScoarse$ is a variety of general type.
\end{prop}
\par
\medskip
\paragraph{\textbf{The canonical class}}
We define the rational number 
\be \label{eq:defkappamu}
\kappa_\mu \= \sum_{ m_i\neq -1} \frac{m_i(m_i+2)}{m_i+1} 
\= 2g - 2 + s + \sum_{m_i\neq -1} \frac{m_i}{m_i+1}\, 
\ee
for any signature $\mu = (m_1, \ldots, m_n)$, where $s$ is the number of
entries equal to~$-1$.
The values  $\kappa_{\mu_{\Gamma}^{\bot}}$ and $\kappa_{\mu_{\Gamma}^{\top}}$ 
are similarly defined for the bottom and top level strata of $\Gamma$, \emph{including the edges
  as legs}. Following the convention in \cite{Kodstrata}, we sometimes simply write $\kappa^\bot$ and $\kappa^\top$ when there is no confusion in the context.  We also denote by $\lambda_1$ the first Chern class of the Hodge bundle pulled back to the moduli space of multi-scale differentials. 
\begin{prop}[{\cite[Proposition~6.2]{Kodstrata}}]
\label{prop:canonicaloncoarse}
The divisor class of the canonical bundle of the coarse moduli space
$\bP\MScoarse[2g-2]^{\even}$  is given by 
\ba \label{eq:canformula}
\frac{\kappa_{(2g-2)}}{2g} &\c_1\bigl(K_{\bP\MScoarse[2g-2]^{\even}}\bigr) 
\=  12 \lambda_1 - \Bigl(1 + \frac{\kappa_{(2g-2)}}{2g}\Bigr)[D_h] \\
& - \sum_{\Gamma\in \LG_1} \Big(\ell_\Gamma  \kappa_{\mu_{\Gamma}^{\bot}}
- \frac{\kappa_{(2g-2)}} {2g} (\ell_\Gamma N_\Gamma^{\bot}-1)
 \Big) [D_\Gamma]
- \frac{\kappa_{(2g-2)}}{2g} \sum_{\Gamma \in \mathrm{HBB}}
 [D^{\rm H}_\Gamma]
\ea
in $\CH^1(\bP\MScoarse[2g-2]^{\even}) = \CH^1(\bP\LMS[2g-2][g,1]^{\even})$.
\end{prop}
\par
We remark that the last sum is over ``hyperelliptic banana backbone" (HBB)
graphs, which are level graphs with a single vertex on the lower level
connected to each vertex (arbitrary in number) on the top level either by
a single edge or by a pair (``banana'') of edges with the same enhancement. 
Additionally, each vertex must belong to a hyperelliptic component of the respective stratum, the graph must contain at least one banana, and the prong-matchings are chosen to have even spin but non-hyperelliptic structure (see \cite[Section~2.3]{Kodstrata}).  
\par
\medskip
\paragraph{\textbf{Pullback divisors}}
We recall from \cite[Section~6.2]{Kodstrata} that two divisor classes 
initially used by Harris and Mumford pull back to divisor classes on strata.
We define the pullback \emph{Brill--Noether divisor class $[\BN_\mu]$} (as a $\bQ$-divisor) to be the total
transform  $[\BN_\mu] = f^*[\BN_g]$, where $f\colon \bP\LMS \to\barmoduli[g]$ is
the forgetful map and~$[\BN_g]$ is a rational multiple of the class (of the closure) 
\ba \label{eq:defBN} 
\wt{\BN}_g & \= \{ X \in \moduli[g]\,: X \, \text{has a} \,\,\frakg^1_k\ \text{where}\ k=(g+1)/2\}\,. 
\ea
We refer to~\cite{HarrisMumford} for the original computation of the class of this Brill--Noether divisor in $\overline{\mathcal M}_g$. 
\par 
For an edge~$e$ in a level graph~$\Gamma$, we write $e \mapsto  \Delta_i$ if
contracting all edges of~$\Gamma$ except for~$e$ results in a graph of compact type 
parametrized by the boundary divisor~$\Delta_i$ in $\barmoduli[g]$. We
write $e \mapsto \Delta_{\irr}$ if the edge is non-separating.
\par
\begin{lemma}[{\cite[Lemma~6.7]{Kodstrata}}]
\label{lem:BNclass} Let $g\geq 3$ be odd. 
The Brill--Noether divisor class in the stratum $\bP\LMS[2g-2][g,1]^{\even}$ is an effective divisor class, which can be expressed as 
$$ [\BN_{(2g-2)}] \= 6\lambda_1 - \frac{g+1}{g+3} [D_h] - \sum_{\Gamma\in \LG_1}
b_\Gamma [D_\Gamma]\,, $$
where 
$$b_\Gamma \= \ell_\Gamma \Big(\sum_{i=1}^{[g/2]}\sum_{e\in E(\Gamma) \atop e \mapsto \Delta_i} \frac{6i(g-i)}{(g+3) p_e} +  
\sum_{e\in E(\Gamma) \atop e \mapsto \Delta_{\irr}} \frac{g+1}{(g+3) p_e}  \Big)\,. $$
\end{lemma}
\par
Similarly, for even genus, we let $[\Hur_\mu] = f^*[\Hur_g]$ be the
pullback {\em Hurwitz divisor class}, where~$[\Hur_g]$ is proportional to the class
(of the closure) 
\ba \label{eq:defHur}
\wt{\Hur}_g &\= \{ X \in \moduli[g]\,: \text{there exists a cover}\,\,
\pi\colon X \to \bP^1, \, \deg(\pi) = (g+2)/2, \\
&  \qquad \qquad \qquad \qquad
\text{$\pi$ has a ramification point of multiplicity three}\,\}\,.
\ea
We refer to~\cite{HarrisKodII} for the original computation of the class of this Hurwitz divisor in $\overline{\mathcal M}_g$. 
\begin{lemma}[{\cite[Lemma~6.8]{Kodstrata}}] \label{lem:Hurclass}
Let $g \geq 6$ be even. The Hurwitz divisor class $[\Hur_\mu]$ in the stratum $\bP\LMS[2g-2][g,1]^\even$ is an effective divisor class, which can be expressed as   
$$ [\Hur_\mu] \= 6\lambda_1 -  \frac{3g^2 + 12g -6}{(g+8)(3g-1)} [D_h] - \sum_{\Gamma\in \LG_1}
h_\Gamma [D_\Gamma]\,, $$
where 
$$h_\Gamma \= \ell_\Gamma \Big(\sum_{i=1}^{g/2}\sum_{e\in E(\Gamma) \atop e \mapsto \Delta_i} \frac{6 i (g-i)(3g+4)}{(g+8)(3g-1) p_e} +  
\sum_{e\in E(\Gamma) \atop e \mapsto \Delta_{\irr}} \frac{3g^2 + 12g -6}{(g+8)(3g-1) p_e} \Big)\,. $$ 
\end{lemma}

%% file: sec_pullback.tex
\section{Compact type clutching for the minimal strata}
\label{sec:CTclutching}

For general level graphs~$\Gamma$, clutching maps have been constructed
in \cite[Section~4.2]{CMZEuler}. These maps are somewhat involved because the
boundary strata are only commensurable with a product. Here, we consider
a special case used to construct effective divisor classes in
Section~\ref{sec:genWP}, where the clutching map simplifies significantly.
Our goal is to provide formulas for the pullback of divisor classes along this
clutching map.
\par
Let $\mu = (m_1,\ldots, m_n)$ be a positive partition of~$2g$. 
Consider the boundary divisor $D_{\Gamma_1}$ corresponding to the level
graph~$\Gamma_1$, 
\begin{center}
	\begin{tikzpicture}[auto, node distance=2cm, every loop/.style={},
		thick,main node/.style={circle,draw,font=\sffamily\Large\bfseries}]
		\node[main node] (1) {$g$};
		\node[main node] (2) [below of =1] {$1$};
		
		\path[every node/.style={font=\sffamily\small}]
		(1) edge node [] {} (2) node [xshift=30,yshift=-35] {$p_e = 2g-1$};
		
		\draw (2) -- +(-135: 1) node [xshift=0,yshift=-10] {$m_1$};
		\draw (2) -- +(-45:1) node [xshift=10,yshift=-5] {$m_n$};
		\draw (2) -- +(-90:1) node [xshift=5,yshift=-5] {$m_2\cdots$};
		
		\node [xshift=20,yshift=-10] {even};
	\end{tikzpicture} 
\end{center}
which is of compact type, with the genus one bottom level carrying all the markings.
\par
\begin{lemma} There is a well-defined, injective clutching map on the
open boundary strata, 
\bes
{\zeta_1}\colon \bP\omoduli[g,1](2g-2)^{\even} \times
\bP\omoduli[1,n+1](-2g, \mu) \rightarrow \bP\LMS[\mu][g+1,n]
\ees
whose range lies in the boundary stratum~$D_{\Gamma_1}$. This clutching map
extends to a map
\bes
\overline{\zeta}_1\colon \bP\LMS[2g-2][g,1]^{\even} \times
\bP\omoduli[1,n+1](-2g, \mu) \rightarrow \bP\LMS[\mu][g+1,n]
\ees
on the multi-scale compactification of the top-level factor.
\end{lemma}
\par
\begin{proof} Well-definedness and injectivity in the first statement hold 
because all the prong-matchings are equivalent for graphs with just one edge.  
\par
The obstruction to extending $\zeta_1$ to the boundary stems from assigning well-defined
prong-matching equivalence classes and handling the local
orbifold structure given by the ghost group, which is the quotient of the twist group
by the simply twist group (see \cite[Section~2.1]{Kodstrata}). We observe that
in any boundary divisor~$\Pi$ of $\bP\LMS[2g-2][g,1]^{\even}$, the lowest level contains a unique vertex,
and the half-edge representing the unique zero is connected to it. Let $\Pi_1$ be the level graph obtained
by extending~$\Pi$ with the lower level of~$\Gamma_1$. We note that
in~$\Gamma_1$, no edge becomes long. This implies that the prong-matching
equivalence classes for~$\Gamma$ and~$\Gamma_1$ are in natural bijection.
\par 
Moreover, the ghost group $K_{\Gamma_1} = \Tw[\Gamma_1]/\sTw[\Gamma_1]$ is trivial, as it is for any two-level graph. More generally, starting with any~$\Pi$,
there is a natural isomorphism of ghost groups $K_{\Pi} \cong K_{\Pi_1}$ since there is no long edge adjacent to the newly created lowest level vertex.
This means that there are no obstructions to extending~$\zeta_1$: we simply join a multi-scale differential from  $\bP\LMS[2g-2][g,1]^{\even}$, which is compatible with~$\Pi$, and a differential from $\omoduli[1,n+1](-2g, \mu)$ to a multi-scale differential compatible with~$\Pi_1$ by using an arbitrary prong-matching along the edge stemming from~$\Gamma_1$.
\end{proof}
\par
We consider a generic element $(E, q, z_1,\ldots, z_n) \in \bP\omoduli[1,n+1](-2g,
\mu)$ and denote the restriction of the clutching map $\overline{\zeta}_1$ by
\bes
{\zeta_E}\colon \bP\LMS[2g-2][g,1]^{\even}  \rightarrow \bP\LMS[\mu][g+1,n] \,.
\ees
For graphs $\Delta \neq \Gamma_1$ corresponding to boundary divisors of
$\bP\LMS[\mu][g+1,1]$, we define $\zeta_E^*\Delta$ in the following way:
\begin{enumerate}
	\item If not all half-edges $z_1, \ldots, z_n$ are on the same vertex at level $-1$ of $\Delta$, then we define $\zeta_E^*\Delta \coloneqq 0$.
	\item If all half-edges $z_1,\ldots, z_n$ are on the same vertex at level $-1$ of $\Delta$, then we define $\zeta_E^*\Delta$ to be the level graph obtained from $\Delta$ by decreasing the genus of the lower level by one and replacing the half-edges $z_1, \ldots, z_n$ with $z$. 
\end{enumerate}
\begin{prop} \label{prop: pullback} The map induced by $\zeta_E$ at the level of Picard groups 
\[\zeta_E^*\colon \mathrm{Pic}(\bP\LMS[\mu][g+1,n])
\rightarrow \mathrm{Pic}(\bP\LMS[2g-2][g,1]^{\even}) \]
	satisfies 
	\[ \zeta_E ^*[D_\Delta]  \= [D_{\zeta_E^*\Delta}] \ \mathrm{for} \ \mathrm{all} \ \Delta \neq \Gamma_1\,. \] 
	Moreover 
\[\zeta_E ^*\lambda_1 \= \lambda_1, \quad \zeta_E^*\psi_i \= 0 \ \mathrm{for} \
\mathrm{all} \  1\leq i\leq n, \quad \mathrm{and} \quad \zeta_E^*[D_{\Gamma_1}]
\= -\psi\,.  \]
\end{prop}
\begin{proof}
All curves parameterized by the image of $\zeta_E$ have the half-edges $z_1,\ldots, z_n$ on the elliptic component, which forces this component to be at the lowest level of the associated level graph. In particular, if a graph $\Delta$ does not have all half-edges on the (unique) vertex at the lowest level, then $D_\Delta$ and $\textrm{Im}(\zeta_E)$ do not intersect. 
\par	
If a graph $\Delta \neq \Gamma_1$ has all half-edges on the (unique) vertex at the lowest level, then the curves parameterized by the intersection $D_\Delta \cap D_{\Gamma_1}$ will have the graph $\zeta_E^*\Delta$ at levels~$0$ and~$-1$, glued via the edge $q$ to an elliptic curve containing $z_1, \ldots, z_n$ at level~$-2$.  
\par	
Because the boundary divisors of $\mathbb{P}\Xi\overline{\mathcal{M}}_{g+1,n}(\mu)$ are normal crossing and the map $\zeta_E$ injects into $D_{\Gamma_1}\subseteq \mathbb{P}\Xi\overline{\mathcal{M}}_{g+1,n}(\mu)$, it follows that when we pull back boundary divisors $D_\Delta$ by $\zeta_E$, they will appear with multiplicity $1$ (or $0$ if the divisors $D_{\Gamma_1}$ and $D_\Delta$ do not intersect). Together, these facts imply 
	\[ \zeta_E^*[D_\Delta]  \= [D_{\zeta_E^*\Delta}] \ \mathrm{for} \ \mathrm{all} \ \Delta \neq \Gamma_1\,. \]
\par 	
To compute the pullback of the other divisor classes, we consider the diagram
	\[
	\begin{tikzcd}
\bP\LMS[2g-2][g,1]^{\even} \ar[d] \arrow{r}{\zeta_E} & \bP\LMS[\mu][g+1,n]  \ar[d] \\
\mm_{g,1}\arrow{r}{}& \mm_{g+1,n}
	\end{tikzcd}
	\] 	
The pullbacks at the level of Picard groups are well understood for all
maps in the diagram except for~$\zeta_E$; see \cite[Section 2.7]{MulMark} or \cite{cornintersection} for the map between the moduli spaces of curves, and \cite[Section~6]{Kodstrata} for the vertical maps. Therefore, we obtain 
\[ \zeta_E^*[D_{\Gamma_1}] \= -\psi, \quad \zeta_E^*\lambda_1 \= \lambda_1, \quad
\ \mathrm{and} \quad  \zeta_E^*\psi_i \= 0  \]
for $1\leq i\leq n$. 	
\end{proof}
\par
We recall that the class $\xi = c_1(\mathcal{O}(-1))$, the first Chern class of the
tautological bundle on $\bP\LMS[\mu][g+1,n]$, pulls back to the respective $\xi$-class 
 on the top level of any boundary stratum. Therefore, we can write with the intended ambiguity
\be \label{eq:xipullback}
\zeta_E^*\xi \= \xi\,.
\ee
(The reader can further verify the above relation by using Proposition~\ref{prop: pullback} and
\be \label{eq:xiconversion}
\xi \= (2g-1)\psi- \sum_{\Delta\in {\LG_1}} \ell_\Delta[D_\Delta]
\ee
from \cite[Formula (38)]{Kodstrata}.)

%% file: sec_BN.tex
\section{Limit linear series and chains of divisors}
\label{sec:BN}

Our goal is to use the map~$\zeta_E$ to pull back generalized Weierstrass divisors and obtain effective
divisors on $\bP\LMS[2g-2][g,1]^{\even}$. To ensure that the resulting divisor  on
$\bP\LMS[2g-2][g,1]^{\even}$ is effective, we need to verify that the divisor we pull
back is not supported on the image of $\zeta_E$. By defining a 
generalized Weierstrass divisor as the degeneracy locus of a vector bundle
morphism (see~\eqref{eq:defgenWG} below), our task reduces to determining the exact order of vanishing along the divisor $D_{\Gamma_1}$. We will accomplish this by applying the theory of limit linear series to study vanishing orders of limit canonical series on various chains of divisors. 

\subsection{Limit linear series} 

In their fundamental work on Brill--Noether Theory (see \cite{limitlinearbasic}), Eisenbud and Harris aimed to understand line bundles of a given degree with many global sections. 
We briefly recall the related definition and notation. A linear series $g^r_d$ on $X$ is a pair $\ell = (L,V)$, where $L$  is a line bundle of degree $d$ on $X$ and $ V\subseteq H^0(X,L)$ is an $(r+1)$-dimensional subspace of the space of global sections of $L$. The variety parameterizing all $g^r_d$'s on a curve $X$ is denoted by $G^r_d(X)$. 
\par
Given distinguished points $z_1, \ldots, z_n\in X$, it is natural to examine the vanishing orders of the linear series at these points. We begin with the classical definition from the literature, concerning a unique point $z\in X$ (see \cite[Section 0]{limitlinearbasic}), and then provide the more general definition of multi-vanishing orders, as found in \cite[Definition 1.1]{Ossermandim} and \cite[Definition 4.2]{Ossermancompacttype}. 
\par
\begin{definition}	Let $z \in X$, and let $\ell = (L, V)$ be a $g^r_d$ on $X$. Consider the vanishing orders of the sections of $V$ at the point $z$. Since $V$ is $(r+1)$-dimensional,
there are exactly $r+1$ distinct vanishing orders, which, when sorted in strictly increasing order, 
form the \emph{vanishing sequence}
\[a^{\ell} (z): 0\leq a_0^{\ell}(z) < a_1^{\ell}(z) < \cdots < a_r^{\ell}(z) \leq d\]
with respect to~$z$.
\par
More generally, consider a chain $\textbf{D}$ of effective divisors on $X$:  
\[ 0 \= D_0 < D_1 <\cdots < D_{d} \]
satisfying $\deg(D_i) = i$ for every $0\leq i\leq d$. We say that a section $s \in V$ has \emph{multi-vanishing order~$i$} with respect to $\textbf{D}$ if 
\[ s \in V\cap H^0(C,L-D_i) \ \textrm{and} \ s \not\in V\cap H^0(C,L-D_{i+1}), \ \textrm{when} \ 0\leq i \leq d-1, \ \textrm{or} \]
\[ s \in V\cap H^0(C,L-D_i) \ \textrm{when} \ i = d.\]  
As before, there are exactly $r+1$ multi-vanishing orders, which form a
\emph{multi-vanishing sequence}
\[a^{\ell} (\textbf{D}): 0\leq a_0^{\ell}(\textbf{D}) < a_1^{\ell}(\textbf{D}) < \cdots < a_r^{\ell}(\textbf{D}) \leq d\]
with respect to~$\textbf{D}$.
\end{definition}
\par 
Vanishing orders at points play an important role in understanding linear series as smooth curves degenerate to a nodal curve of compact type (i.e., one whose dual graph is a tree). To briefly describe how to interpret the linear series in the limit, consider a generically smooth family of curves 
\[ \pi\colon X \rightarrow B\]
with a family $(L_b, V_b)$ of linear series $g^r_d$ for $b\neq 0$, and a central fiber $X_0$ of compact type. The line bundle $L$ over $X \setminus X_0$ does not extend uniquely over $X_0$: it is only unique up to twisting with the components of the central fiber $X_0$. Since $X_0$ is of compact type, we can twist the line bundle in such a way that its degree is $d$ on a chosen irreducible component and $0$ on all the others. By also considering how the sections degenerate, we obtain a $g^r_d$ for each irreducible component of $X_0$, subject to vanishing conditions at the nodes (see \cite[Proposition 2.2]{limitlinearbasic}). This motivates the definition of limit linear series: 
\par
\begin{definition} \label{def:limit linear series} Let $X$ be a curve of compact type. A \emph{limit linear series $g^r_d$} on $X$ is a collection of linear series, one for each irreducible component
of $X$:   
\[\ell \= \{\ell_Y \=(L_Y,V_Y)\in G^r_d(Y) \ | \ Y\subseteq X \text{ is an irreducible component}\}\,,\]
	satisfying the following compatibility conditions: if $Y$ and $Z$ are irreducible components of $X$ intersecting at a node $q=Y\cap Z$, then \[a^{\ell_Y}_i(q) + a^{\ell_Z}_{r-i}(q) \geq d \quad \text{ for all $0\leq i \leq r$.}\]  
	The pair $\ell_Y= (L_Y, V_Y)$ is called the \emph{$Y$-aspect} of the limit linear series $\ell$. If, for every node, all inequalities are equalities, the limit linear series is called \emph{refined}. 
\end{definition}
\par
This definition can be extended to linear series with multi-vanishing orders with respect to a chain of divisors. 
\par
\begin{definition}
If $X$ is a curve of compact type and $\textbf{D}$ is a chain of divisors supported entirely on an irreducible component $Z$, a \emph{limit linear series $g^r_d$ on $X$ having multi-vanishing orders with respect to $\textbf{D}$ greater than or equal to $a = (a_0,\ldots, a_r)$} is a limit linear series $\ell$ as defined in Definition~\ref{def:limit linear series} that further 
satisfies 
	\[ a_i^{\ell_Z}(\textbf{D}) \geq a_i \quad \textrm{for all} \ 0\leq i \leq r.\]	
\end{definition}
\par
We are mainly interested in the case where $r = g-1$ and $d = 2g-2$, and want to understand the space of canonical linear series 
\[ G^{g-1}_{2g-2}(X) = \left\{(\omega_X, H^0(C, \omega_X))\right\}\]
as $X$ degenerates to the boundary of the moduli space. More precisely, we would like to understand
the limit canonical series on a generic element of $D_{\Gamma_1}$. By generalizing
the methods used in \cite{Cukierman}, we will show that
these limit canonical series are refined, and the vanishing orders
are $(0,1,\ldots, g)$ with respect to a chain of divisors supported on the marked points.  
\par
%
%
%
To achieve this, we will rely on the work of Bullock concerning the canonical
series $\ell = (\omega_X, H^0(X,\omega_X))$  on a curve $(X,z) \in \omoduli[g](2g-2)$.
\par
\begin{theorem}\normalfont{(\cite[Main Theorem]{bullocksubcanonical})}
\label{thm:bullock}
Let $(X,z) \in \omoduli[g](2g-2)$. Then the vanishing orders of the canonical
series at $z$ are 
	\begin{itemize}
\item $(0,2, 4,\ldots, 2g-2)$ if $(X,z)$ is generic in
$\omoduli[g](2g-2)^{\textrm{hyp}}$; 
\item $(0,1,2, \ldots, g-2, 2g-2)$  if $(X,z)$ is
generic in $\omoduli[g](2g-2)^{\textrm{odd}}$;  
\item $(0,1,2, \ldots, g-3,g-1, 2g-2)$  if $(X,z)$ is generic in $\omoduli[g](2g-2)^{\textrm{even}}$.
	\end{itemize}
\end{theorem} 
\par
To obtain this result, Bullock uses parity-preserving clutching maps 
\bes
{\zeta^+_E}\colon \bP\LMS[2g-4][g-1,1]^{\textrm{even}}  \rightarrow \bP\LMS[2g-2][g,1]^{\textrm{even}} \,
\ees
and
\bes
{\zeta^-_E}\colon \bP\LMS[2g-4][g-1,1]^{\textrm{odd}}  \rightarrow \bP\LMS[2g-2][g,1]^{\textrm{odd}} \,,
\ees
similar to the ones considered in Section~\ref{sec:CTclutching}. For a generic element $(X_0\cup_qE,z)$ in the image, Bullock computes the vanishing orders of the $E$-aspect
at both the node~$q$ and the marked point~$z$ (see \cite[Lemma~2.2]{bullocksubcanonical}).
Proposition \ref{prop:even-general-vanishing} below extends this result to curves
$(X_0\cup_qE,z_1,\ldots, z_n)$ in the image of 
\bes
{\zeta_E}\colon \bP\LMS[2g-2][g,1]^{\even}  \rightarrow \bP\LMS[\mu][g+1,n] \,, 
\ees
for any $n\geq 2$. 
\par
In order to generalize Cukierman's argument in \cite{Cukierman}, we need to understand, for the $E$-aspect of the limit canonical series, the interplay between the vanishing orders with respect to~$q$ and with respect to a chain of divisors $\textbf{D}_\bullet(z_1,z_2,\ldots, z_n)$ supported on the points $z_1, \ldots, z_n$. This is described in greater generality in the following lemma, which relates the vanishing orders for two divisorial chains.
\par
\begin{lemma} \label{lemma:two ramification conditions}
Let $X$ be a smooth genus $g$ curve, and let $z_1, \ldots, z_m$ and $q_1,\ldots, q_n$ be
two collections of points on~$X$. Consider two increasing sequences of effective divisors:  
	\[\textbf{D}_\bullet(z_1,z_2,\ldots,z_m) = (0 = D_0 < D_1 < D_2<\cdots < D_d) \]
	and 
	\[\textbf{E}_\bullet(q_1,q_2,\ldots,q_n) = (0 = E_0 < E_1 < E_2<\cdots < E_d)\]
satisfying $\deg(D_i) = \deg(E_i) = i$ for all $0\leq i \leq d$, where all $D_i$'s are supported on the points $z_1, \ldots, z_m$ and all $E_i$'s are supported on the points $q_1, \ldots, q_n$. 
\par	
	Let $(V, L)$ be a $g^r_d$ on $X$ having vanishing orders $(a_0,\ldots, a_r)$ and $(b_0, \ldots, b_r)$ with respect to $\textbf{D}_\bullet$ and $\textbf{E}_\bullet$. Then there exists a basis $s_0,\ldots,s_r$ of $V$ and a permutation~$\sigma$ such that 
	\[ \mathrm{ord}_{\textbf{D}_\bullet}(s_i) = a_i \quad \mathrm{and} \quad \mathrm{ord}_{\textbf{E}_\bullet}(s_i) = b_{\sigma(i)} \quad \forall \ 0\leq i \leq r\,.\] 
\end{lemma}
\begin{proof}
The proof of this proposition is identical to that of the classical case when $m = n = 1$.  Let $s_0, s_1, \ldots, s_r$ be a basis of $V$ satisfying $\mathrm{ord}_{\textbf{D}_\bullet}(s_i) = a_i $ for every $0\leq i \leq r$. Assume that two of these sections, say $s_i$ and $s_j$ with $i< j$, have the same vanishing order $k$ with respect to $\textbf{E}_\bullet$. Let $q = E_{k+1} - E_k$ and consider the differences $s_i - xs_j$ for $x \in \mathbb{C}$. There exists a choice of $x$ such that $s_i - xs_j$ vanishes to order at least one more at $q$, which implies $\textrm{ord}_{\textbf{E}_\bullet}(s_i) < \textrm{ord}_{\textbf{E}_\bullet}(s_i-xs_j)$, while $s_i - xs_j$ and $s_i$ have the same vanishing order with respect to $\textbf{D}_\bullet$. We then replace the section $s_i$ with $s_i-xs_j$. By repeating this procedure a finite number of times, we obtain a basis satisfying the desired conditions.     
\end{proof}


%
%
\par
As noticed in Theorem \ref{thm:bullock}, the marked point is a very special Weierstrass point with vanishing orders that differ significantly from the generic $(0, 1, \ldots, g)$. We want to restrict ourselves to chains of divisors that are supported on the marked points, with ramification orders that are generic. For this to happen, we need $\sum_{i=1}^n m_iz_i - D_{g}$ to be effective, while $\sum_{i=1}^n m_iz_i - D_{g+1}$ is not. This motivates the following definition. 
\par
\begin{definition}
Let $\mu = (m_1, \ldots, m_n)$ be a positive partition of $2g$ and $(X,\mathbf{z}, \omega)$ an element of $ \bP\LMS[\mu][g+1,n]$. Consider a chain of effective divisors  
\[\textbf{D}_\bullet(\textbf{z}) = (0 = D_0 < D_1 < D_2<\cdots < D_{2g})\] such that all divisors are supported on the points $z_1,\ldots, z_n$ and 
the degree of $D_i$ is $i$ for every $0 \leq i \leq 2g$. Given $1\leq k \leq n$, we say that $\textbf{D}_\bullet(\textbf{z})$ is \emph{$k$-saturated} if it satisfies the following conditions: 
	\begin{itemize}
		\item For each $i \geq g$, the point $z_k$ appears in the support of $D_i$ with order exactly $(i-g)+ m_k$. 
		\item For each $0\leq i\leq 2g$ and each $j\neq k$, the point $z_j$ appears in the support of $D_i$ with order less than or equal to $m_j$. 
	\end{itemize}
\end{definition}
\par
We consider multi-vanishing orders with respect to $k$-saturated chains of divisors. Our goal is to show that for a general element of $D_{\Gamma_1}$, the multi-vanishing orders of the limit canonical series with respect to some chain $\textbf{D}_\bullet(\textbf{z})$ are $(0,1,\ldots, g)$, see Proposition \ref{prop:even-general-vanishing}. As in \cite{Cukierman} and \cite{limitlinearbasic}, knowing the exact multi-vanishing orders will allow us to pull back generalized Weierstrass divisors via the map $\zeta_E$. In some instances, we can obtain a geometric description of the pullback divisor, as described in Corollary \ref{cor: obtaining Weierstrass divisor}. In order to apply this corollary, we need to choose a partition $\mu$ and a chain of divisors $\textbf{D}_\bullet(\textbf{z})$ for which the hypothesis of Corollary \ref{cor: obtaining Weierstrass divisor} is satisfied. This will be the content of Proposition \ref{prop: even cases}.     
\par
 To show the multi-vanishing orders with respect to certain chain of divisors are generic, we assume by contradiction that they are not. Having non-generic vanishing orders with respect to the concerned chains of divisors is a very strong condition, and we will immediately obtain a contradiction from the existence of a section for the $E$-aspect with unattainable vanishing orders. 
\par
In order to simultaneously control the vanishing orders with respect to two distinct chains of divisors, we will require that they agree on the degree $g-1$ divisor.
\par
\begin{prop} \label{prop:even-general-vanishing}
	Let $(X_0\cup_qE, \textbf{z}, \omega)$ be a generic element in the boundary divisor $D_{\Gamma_1}$. Given two distinct indices $1\leq j\neq k \leq n$, suppose $\textbf{D}^j_\bullet(\textbf{z})$ and $\textbf{D}^k_\bullet(\textbf{z})$  are $j$-saturated and $k$-saturated chains of divisors, respectively, where $\textbf{D}^j_{g-1} = \textbf{D}^k_{g-1}$. Then, there exists an index $l\in \{j,k\}$ such that the vanishing sequence with respect to  $\textbf{D}^l_\bullet(\textbf{z})$ for the $E$-aspect of the limit canonical series on $X_0\cup_qE$ is $(0,1,2,\ldots, g)$. Additionally, the limit canonical series on this curve is refined. 
\end{prop}
\begin{proof}
Suppose that the vanishing orders with respect to both $\textbf{D}^j_\bullet(\textbf{z})$ and $\textbf{D}^k_\bullet(\textbf{z})$ are not $(0,1,\ldots, g)$. We will deduce a contradiction.
 \par
Because $(X_0,q)$ is a generic element of $\omoduli[{g}](2g-2)^{\textrm{even}}$, the vanishing orders of $H^0(X_0, \omega_{X_0}(2q))$ at $q$ are $(0,2,3,4,\ldots,g-1, g+1, 2g)$ by Theorem~\ref{thm:bullock}. Using the inequalities in Definition \ref{def:limit linear series}, the vanishing orders at $q$ of the $E$-aspect are greater than or equal to $(0,g-1,g+1, \ldots, 2g-2, 2g)$. Since the last entry cannot be greater than $2g$, it must be equal to $2g$. 
\par
Moreover, the data of an element $(X_0\cup_qE, \mathbf{z}, \omega)$ in $D_{\Gamma_1}$ includes a collection of (possibly meromorphic) differentials, one for each irreducible component. The differential $\varphi_E$ associated with $E$ satisfies 
\[\varphi_E \in H^0\Bigl(E, \omega_E\Big(2gq- \sum_{i=1}^nm_iz_i\Big)\Bigr)\,.\] 
By viewing $\varphi_E$ as a section of $H^0(E, \omega_E(2gq))$ via the obvious inclusion, it becomes a section of the $E$-aspect of the limit canonical series: As a smooth curve $X$ degenerates to $X_0\cup_qE$, the section $\omega \in H^0(X, \omega_X)$ degenerates to a section whose $E$-aspect is $\varphi_E$. 
\par 
In terms of a $k$-saturated chain of divisors, if we denote $D_g = \sum_{i=1}^{n}\alpha_iz_i$, where $\alpha_i \leq m_i$ for all $i$ and $\alpha_k = m_k$, we have
 \[\textrm{div}(\varphi_E) \= D_{g} + \sum_{i=1}^{n}(m_i-\alpha_i)z_i\,.\] 
 Since $m_k-\alpha_k = 0$, it follows that $\varphi_E$ vanishes to order~$g$ with respect to $\textbf{D}_\bullet(\textbf{z})$. Moreover, $\varphi_E$ does not vanish at~$q$. Thus, we conclude that $0$ is a vanishing order with respect to $q$, and that~$g$ is a vanishing order with respect to $\textbf{D}_\bullet(\textbf{z})$. Applying this argument to the two saturated chains of divisors in the hypothesis yields that $0$ is a vanishing order at $q$ and that $g$ is a vanishing order with respect to both chains. 
 \par 
 Combining the above with our assumption that the vanishing orders with respect to $\textbf{D}^k_\bullet(\textbf{z})$ are not $(0,1,\ldots, g)$, the vanishing orders with respect to $\textbf{D}^k_\bullet(\textbf{z})$ must then be greater than or equal to  $ (0,1,\ldots, g-2, g, g+1)$. Using Lemma \ref{lemma:two ramification conditions}, the vanishing orders of the $E$-aspect at $q$ are less than or equal to $(g-2, g-1, g+1, g+2,\ldots, 2g-2, 2g)$. To see this, if one of the orders $i$ is greater than the claimed value, there must be two independent sections of $H^0\big(E, \omega_E(2gq)\big)$ as in Lemma \ref{lemma:two ramification conditions}, whose vanishing orders with respect to $\textbf{D}^k_\bullet (\textbf{z})$ and $q$ add up to $2g$. In particular, we would have the equivalence of divisors
 $ iq + D_{2g-i} \sim 2gq$ 
 for some $i\in \left\{g-1, g,g+2,g+3,\ldots 2g-1\right\}$, which is false for a generic $[E, q,z_1,\ldots, z_n,\varphi_E]$ in the stratum $\omoduli[1](-2g, m_1, m_2, \ldots, m_n)$. 
\par
Now we have two inequalities for the vanishing orders at $q$, and we know that $0$ is one of the orders, which comes from that of the section $\varphi_E$. These constraints on the vanishing orders determine a unique possibility that the vanishing orders of the $E$-aspect at $q$ are $(0, g-1, g+1, g+2, \ldots, 2g-2, 2g)$. 
\par

Reasoning analogously for the chain of divisors $\textbf{D}^j_\bullet(\textbf{z})$, the vanishing orders with respect to $\textbf{D}^j_\bullet(\textbf{z})$ are greater than or equal to  $ (0,1,\ldots, g-2, g, g+1)$. 

Using Lemma \ref{lemma:two ramification conditions} for the $E$-aspect of the limit canonical series, we obtain the existence of two sections $\sigma_1$ and $\sigma_2$ satisfying 
\begin{align*}
	\textrm{div}(\sigma_1) &= (g-1)\cdot q + D_g^j + x_1 \ \textrm{for some point} \ x_1 \in E \ \textrm{and} \\
	\textrm{div}(\sigma_2) &= (g-1)\cdot q + D_g^k + x_2 \ \textrm{for some point} \ x_2 \in E.
\end{align*}
Notice that the two divisors are not equal, i.e., the sections are not proportional. Otherwise, the equality $D_g^j + x_1 =  D_g^k + x_2$ forces $x_1 = z_k$ and $x_2 = z_j$. This immediately implies that $(g+1) \cdot q \sim D_g^k + z_j$, which is false for a generic $[E, q,z_1,\ldots, z_n,\varphi_E]$ in the stratum $\omoduli[1](-2g, m_1, m_2, \ldots, m_n)$. 

Since the two sections are not proportional, there exists a constant $t \in \mathbb{C}$ so that the vanishing order of $\sigma_1 - t\cdot \sigma_2$ at $q$ is strictly greater than $g-1$. Because the vanishing orders of the limit canonical series at $q$ are $(0,g-1,g+1, g+2,\ldots, 2g-2,2g)$, it follows that  $\sigma_1 - t\cdot \sigma_2$ vanish at $q$ with order at least $g+1$. Since $D^j_{g-1} = D^k_{g-1}$, it follows that $\sigma_1 - t\cdot \sigma_2$ vanishes along this divisor. This implies 
\[ \textrm{div}(\sigma_1 - t\cdot \sigma_2) = (g+1)\cdot q +D^j_{g-1}, \]
leading to the contradiction that $(g-1) \cdot q \sim D^j_{g-1}$. Hence, our hypothesis at the beginning of the proof was wrong, and the first part of the proposition thus follows. 

It remains to prove that the limit canonical series on $(X_0\cup_qE, \textbf{z}, \omega)$ is refined. To see this, consider the elliptic curve $(E\cup_p\mathbb{P}^1, q, z_1, \ldots, z_n)$, where $p-q$ is $2g$-torsion on $E$ and $z_1, \ldots, z_n \in \mathbb{P}^1$. We know that the limit canonical series on the curve $(X_0\cup_qE\cup_p \mathbb{P}^1, \textbf{z}, \omega)$ is refined: Indeed, the vanishing orders of the $E$-aspect at the point $p$ are greater than or equal to $(0,1,\ldots, g-2,g,2g)$ (see Theorem \ref{thm:bullock}). Using Lemma \ref{lemma:two ramification conditions}, the vanishing orders of the $E$-aspect at the point $q$ are less than or equal to $(0, g-1, g+1, \ldots, 2g-2, 2g)$. But the vanishing orders of the $X_0$-aspect at $q$ are equal to $(0,2,\ldots, g-1, g+1, 2g)$. The inequality in the definition of limit linear series implies that the vanishing orders of the  $E$-aspect at the point $q$ are greater than or equal to $(0, g-1, g+1, \ldots, 2g-2, 2g)$, hence the equality. In particular, this implies that the limit canonical series on $(X_0\cup_qE\cup_p \mathbb{P}^1, \textbf{z}, \omega)$ is refined, as desired. Because being refined is an open condition, the same is true for $(X_0\cup_qE, \textbf{z}, \omega)$ as a generic element in the boundary divisor $D_{\Gamma_1}$ of $\bP\LMS[\mu][g+1,n]$.

\end{proof}

We remark that we have only used the genericity of $(X_0,q)$ in  $\omoduli[{g}](2g-2)^{\textrm{even}}$ to conclude that the vanishing orders of the $X_0$-aspect at $q$ are $(0,2,3,\ldots, g-1, g+1, 2g)$. Thus, a corollary of the previous proposition is the following: 

\begin{cor} \label{cor: obtaining Weierstrass divisor}	Let $(X_0\cup_qE, \textbf{z}, \omega)$ be an element in the boundary divisor $D_{\Gamma_1}$ with $(E,q,\textbf{z})$ generic in $\omoduli[1](-2g, m_1, m_2, \ldots, m_n)$. Given two distinct indices $1\leq j\neq k \leq n$, let $\textbf{D}^j_\bullet(\textbf{z})$ and $\textbf{D}^k_\bullet(\textbf{z})$  be $j$-saturated and $k$-saturated chains of divisors, respectively, where $\textbf{D}^j_{g-1} = \textbf{D}^k_{g-1}$. Consider a limit canonical series on $(X_0\cup_qE, \textbf{z}, \omega)$ and assume the vanishing sequence of the $E$-aspect is greater than or equal to $(0,1,\ldots, g-2, g, g+1)$ with respect to both $\textbf{D}^j_\bullet(\textbf{z})$ and  $\textbf{D}^k_\bullet(\textbf{z})$. Then, the vanishing sequence of the $X_0$-aspect at $q$ is strictly greater than $(0,2,3,\ldots, g-1, g+1, 2g)$. In particular, this implies that $h^0(X_0, g\cdot q) \geq 3$. 
\end{cor}
\par
\begin{proof}
  The condition that the vanishing sequence of the $X_0$-aspect at $q$ is strictly greater than $(0,2,3,\ldots, g-1, g+1, 2g)$ is equivalent to the vanishing sequence being greater than or equal to either $(0,2,3,\ldots, g-2, g,g+1, 2g)$ or $(0,2,3,\ldots, g-1, g+2, 2g)$. These conditions both imply $h^0(X_0, g\cdot q) \geq 3$.
\end{proof}
  \par 
\par
Proposition~\ref{prop:even-general-vanishing} motivates the following definition:
\begin{definition}
\label{def:good-D}
Let $(X_0\cup_qE, \textbf{z}, \omega)$ be a generic element in the boundary divisor $D_{\Gamma_1}$ of a stratum  $ \bP\LMS[\mu][g+1,n]$, and let 
\[\textbf{D}_\bullet(\textbf{z}) = (0 = D_0 < D_1 < D_2<\cdots < D_{2g})\]
be a $k$-saturated chain of divisors. We say that $\textbf{D}_\bullet(\textbf{z})$ is \emph{good} if the vanishing orders of the limit canonical series with respect to it are $(0,1,2,\ldots, g)$. 
\end{definition}
\par
In the next section, we will use such chains of divisors to determine the orders of vanishing of generalized Weierstrass divisors along $D_{\Gamma_1}$. This allows us to obtain effective divisors on $\bP\LMS[2g-2][g,1]^{\even}$ via pullback. Moreover, as a consequence of Corollary~\ref{cor: obtaining Weierstrass divisor}, we have a geometric description of the pullback divisor: It is the Weierstrass divisor with extra vanishing (see Subsection~\ref{sec:VO}). 
\par 
Next, we want to find a setting in which the assumption of Corollary \ref{cor: obtaining Weierstrass divisor} is satisfied, i.e., we want to find a saturated chain with respect to which the vanishing orders are generically $(0,1,\ldots, g-2, g, g+1)$. For this, we consider the case when the partition $\mu$ is even, where one of its entries is~$0$. Namely, let $\mu = (2m_1, 2m_2, \ldots, 2m_n, 0)$ be an even partition and let $(E, q, z_1,\ldots, z_n,p)$ be a generic element in $\bP\omoduli[1,n+2](-2g, \mu)$. We consider the clutching map
\bes
{\zeta_E}\colon \bP\LMS[2g-2][g,1]^{\even}  \rightarrow \bP\LMS[\mu][g+1,n+1]^{\even} \,,
\ees
and consider an $(n+1)$-saturated chain of divisors 
\[\textbf{D}_\bullet(\textbf{z}) \= (0 = D_0 < D_1 < D_2<\cdots < D_{2g})\]
satisfying
\be \label{eq:saturatedchainconditions}
D_g \= \sum_{i=1}^n m_i z_i \quad \text{and} \quad D_{g+j} = \sum_{i=1}^n m_i z_i + j\cdot p
\ee
for any $j\geq 1$. We can then determine the vanishing orders of the limit canonical series with respect to this chain of divisors.
\par
\begin{prop} \label{prop: even cases}
Let $(X_0\cup_qE, \textbf{z}, \omega)$ be a generic element in the boundary divisor $D_{\Gamma_1}$ of a stratum  $ \bP\LMS[\mu][g+1,n+1]^{\even}$, and let  $\textbf{D}_\bullet(\textbf{z}) $ be an $(n+1)$-saturated chain of divisors satisfying~\eqref{eq:saturatedchainconditions}. Then the vanishing orders of the limit canonical series with respect to  $\textbf{D}_\bullet(\textbf{z})$ are $(0,1,2,\ldots, g-2, g, g+1)$.
\end{prop}
\par
\begin{proof}
For a generic element $(X, \textbf{z}, \omega)$ of $ \bP\LMS[\mu][g+1,n+1]^{\even}$, the vanishing orders with respect to the chain $\textbf{D}_\bullet(\textbf{z}) $ are exactly $(0,1,2,\ldots, g-2, g, g+1)$. In particular, when we degenerate to an element $(X_0\cup_qE, \textbf{z}, \omega)$ in the boundary divisor $D_{\Gamma_1}$, the vanishing orders of the limit canonical series with respect to  $\textbf{D}_\bullet(\textbf{z})$ are at least
$(0,1,2,\ldots, g-2, g, g+1)$.

Because $(X_0,q)$ is a generic element of $\omoduli[{g}](2g-2)^{\textrm{even}}$, the vanishing orders of  $H^0(X_0, \omega_{X_0}(2q))$ at $q$ are $(0,2,3,4,\ldots,g-1, g+1, 2g)$ (see Theorem \ref{thm:bullock}). Using the inequalities in Definition \ref{def:limit linear series}, the vanishing orders at $q$ of the $E$-aspect are greater than or equal to $(0,g-1,g+1, \ldots, 2g-2, 2g)$. Since the last entry cannot be greater than $2g$, it must be equal to $2g$. 

Using Lemma \ref{lemma:two ramification conditions}, the first $g-1$ vanishing orders of the $E$-aspect with respect to $\textbf{D}_\bullet(\textbf{z})$ are $0,1,\ldots, g-2$: If the last of these entries were strictly bigger than $g-2$, Lemma \ref{lemma:two ramification conditions} would imply the existence of a section vanishing with order $g+1$ at $q$ and $g-1$ at $\textbf{D}_\bullet(\textbf{z})$. This would imply the equality $ (g+1)\cdot q + D_{g-1} \sim 2g\cdot q$, which is false for a generic element $(E, q, z_1,\ldots, z_n,p) \in \bP\omoduli[1,n+2](-2g, \mu)$. 

Moreover, $g$ is a vanishing order, since it is the vanishing order with respect to $\textbf{D}_\bullet(\textbf{z})$ of the differential $\varphi_E$. Hence we have determined $g$ out of the $g+1$ vanishing orders of the $E$-aspect with respect to  $\textbf{D}_\bullet(\textbf{z})$. 

	    For a generic $p \in E$ we have the equality 
		\[ \dim V_E(-D_g - i\cdot p)  = \max\{\dim V_E(-D_g) - i, 0\}.\]
		Since $\dim V_E(-D_g) = 2$, we immediately obtain that $g+1$ is the last vanishing order with respect to $\textbf{D}_\bullet(\textbf{z})$.
\end{proof}

%% file: sec_genWP.tex
\section{Generalized Weierstrass divisors for even spin strata}
\label{sec:genWP}

In this section, we first recall the definition of generalized Weierstrass
divisors used in \cite{Kodstrata} in a convex combination, along with the Brill--Noether divisor as an effective divisor in the application of Proposition \ref{prop:GenTypeCrit}. 
\par 
Afterwards, we shift our focus to computing the vanishing order along $D_{\Gamma_1}$ of the vector bundle morphism that defines the generalized Weierstrass divisor, utilizing consequences from the limit linear series considerations in the previous section. In Theorem~\ref{theorem:class_twisted_Weierstrass}, we will provide a divisorial
class in the effective cone of $\bP\LMS[2g-2][g,1]^{\even}$ that shares many
similarities with the class of the twisted generalized Weierstrass divisor
in $\bP\LMS[2g-2][g,1]^{\odd}$.
\par
Next, we demonstrate that some generalized Weierstrass loci in even spin components
of non-minimal strata are indeed divisors, and we sketch how this can be used
for Kodaira dimension computations in this context. Finally, we provide examples showing 
that these generalized Weierstrass loci are not irreducible in general.

\subsection{Generalized Weierstrass divisors on general strata}
\label{sec:tWP}

Let  $\alpha = (\alpha_1, \ldots, \alpha_n)$ be a partition of $g-1$ such that
$0\leq \alpha_i \leq m_i$ for all~$i$, where $\mu = (m_1,\ldots,m_n)$ with
$\sum_{i=1}^n m_i = 2g-2$ is the zero type of a stratum in genus~$g$. This adapts the
setting from \cite[Section~7]{Kodstrata} to the target of the map~$\zeta_E$.
\par 
Set-theoretically, the \emph{generalized Weierstrass divisor} associated
with~$\alpha$ in the interior of a stratum of type~$\mu$ is given by
\be \label{eq:initdefGenW}
W_{\mu}(\alpha) \= \Bigl\{ (X, \bfz, \omega) \in \bP\omoduli[g,n](\mu)\,:\,
h^0\Bigl(X, \sum_{i=1}^n \alpha_i z_i\Bigr) \geq 2\Bigr \}\,.
\ee
To define it as a subscheme and over the boundary, we consider the following
Porteous-type setup. Let~$\omega_{\rel}$ be the relative dualizing bundle of
the universal curve $\pi\colon \cX \to \bP\LMS[\mu][g,n]$, and
let $S_i\subset \cX$ be the marked sections. Moreover, let $\cH=\pi_*(\omega_{\rel})$
be the Hodge bundle over $\bP\LMS$, and $\cO(-1)$ its tautological subbundle. Both 
$\cH/\cO(-1)$ and 
$$\cF_\alpha \= \pi_*\Big(\omega_{\rel}/\omega_{\rel}\Big(-\sum_{i=1}^n
\alpha_i S_i\Big)\Big)$$ 
are vector bundles of rank~$g-1$. We define
\be
\label{eq:vartheta}
\vartheta_{\mu, \alpha}\= \sum_{i=1}^n \frac{\alpha_i(\alpha_i+1)}{2(m_i+1)}\,.
\ee
The quantities 
$\vartheta_{\mu_\Gamma^{\bot}, \alpha_\Gamma^{\bot}}$ and
$\vartheta_{\mu_\Gamma^{\top},\alpha_\Gamma^{\top}}$ are similarly defined for the bottom and top level
strata of $\Gamma$. However, $\alpha_\Gamma^{\bot}$ and $\alpha_\Gamma^{\top}$ assign
a value of zero to each leg associated with an edge $e\in E(\Gamma)$.
\par
\begin{prop}[{\cite[Proposition~7.1]{Kodstrata}}] \label{prop:Wclass}
Let $\bP\LMS[\mu][g,n]$ denote a stratum that is not of even spin type. Then the degeneracy locus
\be \label{eq:degeneracy}
\wt{W}_\mu(\alpha) \= \{ \mathrm{rank}(\phi) < g-1\} \subset \bP\LMS[\mu][g,n]
\ee
of the map
\be \label{eq:defgenWG}
\phi\colon \cH/\cO(-1) \to \cF_\alpha
\ee
defined by taking principal parts, is an effective divisor. 
Its class is given by
\bas \ 
     [\wt{W}_{\mu}(\alpha)]
& \= \sum_{i=1}^n \frac{\alpha_i(\alpha_i+1)}{2}\psi_i - \lambda_1 + \xi  \nonumber \\
     & \=
\frac{12+12 \vartheta_{\mu, \alpha} - \kappa_\mu}{\kappa_\mu}\lambda_1
- \frac{1+\vartheta_{\mu, \alpha}}{\kappa_\mu} [D_h] \\
& \,-\, \sum_{\Gamma\in\LG_1} \ell_\Gamma \Big( \frac{\kappa_{\mu_\Gamma^{\bot}}}{\kappa_\mu}
(1+\vartheta_{\mu, \alpha}) - \vartheta_{\mu_\Gamma^{\bot}, \alpha_\Gamma^{\bot}}\Big) [D_\Gamma]\,.
\eas 
\end{prop}
\par
For the conversion between the two expressions above, we used the relation of tautological divisor classes: 
 \be
 \label{eq:eta-xi}
\kappa_\mu  \xi \=   12 \lambda_1 - [D_h]
- \sum_{\Gamma \in \LG_1} \ell_\Gamma \kbot [D_\Gamma]\,
\ee
from \cite[Proposition~6.1]{Kodstrata} (for strata without simple poles)
and~\eqref{eq:xiconversion}.
\par
The degeneracy locus $\wt{W}_\mu(\alpha)$ will, in general, have components
supported entirely in the boundary. We would like to remove them, but it is not easy to compute the precise vanishing orders at the boundary. However, some vanishing can always be eliminated by twisting the canonical bundle as follows. Let $\cX \to \bP\LMS$ be a relatively minimal semi-stable model
with smooth total space of the universal family, and let $V = \sum_{i\in I_\Gamma}
s_{i,\Gamma} X_{i,\Gamma}$ be an effective Cartier divisor supported on
some components of the fibers over the boundary divisors~$D_\Gamma$,
with chosen multiplicities $s_i \geq 0$. For every~$\Gamma$, we require
that $s_{i,\Gamma} = 0$ for all top-level components, $s_i \leq \ell_\Gamma$
for all bottom-level components, and more generally $s_i \leq k p_e$ on the
$k$-th rational component (counting from the top) that the edge~$e$ has been
replaced within the passage to the semi-stable model. These conditions are
necessary so that taking principal parts still factors through the quotient
by~$\cO(-1)$ in the following adapted setup. We define the
\emph{twisted generalized Weierstrass divisor} to be the effective divisor
class $\wt{W}_\cL(\alpha)$ of the degeneracy locus as in~\eqref{eq:degeneracy}, 
where $\omega_\rel$ has been replaced everywhere by $\cL=\omega_\rel (-V)$.
\par
Optimizing the $s_{i,\Gamma}$ so that boundary components are removed as much as possible is an integral
optimization problem for a quadratic function. To state the solution, we
define, for a given level graph, 
\be
P  \= \sum_{e \in E} p_e \quad {\rm and} \quad P_{-1} \= \sum_{e \in E} 1/p_e\,
\ee
the sum of the prongs and the sum of their reciprocals.
\par
\begin{lemma}[{\cite[Lemma~7.4 and Corollary~7.6]{Kodstrata}}]  \label{le:optdifference}
There exist coefficients $s_{i,\Gamma}$ such that the difference of coefficients
between the twisted and untwisted Weierstrass divisors, written as a 
sum of $\lambda$, $D_h$, and $D_\Gamma$'s, is at least 
\be \label{eq:Deltaestimate}
\Delta_\cL \wt{W}_\Gamma \,:=\, \big([\wt{W}_\mu(\alpha)] -
[\wt{W}_\cL(\alpha)]\big)_{[D_\Gamma]} \,\geq\, \Bigl\lfloor \frac{\ell_\Gamma} 2 \Bigr
\rfloor  \Bigl(\alpha^\bot - \frac12 m^\bot\Bigr)
+ \frac{\ell_\Gamma}8 \Bigl(P - P_{-1}\Bigr)\,
\ee
for every boundary divisor~$D_\Gamma$, where $\alpha^\bot$ and $m^\bot$
are the sums of the $\alpha$'s and $m$'s at all legs adjacent to the bottom level.
\par
Moreover, the twisted Weierstrass divisor vanishes to order $\ell_\Gamma(v^\top -1)/2$
along~$D_\Gamma$, where $v^\top$ is the number of vertices of~$\Gamma$ on the top level.
\end{lemma}

\subsection{Generalized Weierstrass divisors for even spin strata}
\label{subsec:Weierstrass-nonminimal}
Consider a partition $\mu = (m_1,\ldots, m_n)$ of $2g-2$, where all entries are even, and let $\alpha = (\alpha_1, \ldots, \alpha_n)$ be a partition of $g-1$ satisfying
\bes
\alpha \neq \frac{\mu}{2} \qquad  \text{and} \qquad
0\leq \alpha_i \leq m_i \quad \text{for all $1\leq i \leq n$.}
\ees
\par 

Our goal is to show that both Proposition \ref{prop:Wclass} as well as Lemma \ref{le:optdifference} hold for even spin components of strata, as long as $\alpha \neq \frac{\mu}{2}$. In fact, the methods of \cite{Kodstrata} apply verbatim for the even spin strata, as long as the locus 
\[W_{\mu}(\alpha) \= \Bigl\{ (X, \bfz, \omega) \in \bP\omoduli[g,n](\mu)^{\even}\,:\,
h^0\Bigl(X, \sum_{i=1}^n \alpha_i z_i\Bigr) \geq 2\Bigr \} \] 
is not the entire even spin component. We will prove that it is indeed the case.

\begin{prop} \label{prop:properdivisor}
	Let $\mu$ and $\alpha$ be partitions of $2g-2$ and $g-1$, respectively, as introduced above. The locus $W_\mu(\alpha)$ of elements $(X,\mathbf{z}, \omega) \in \omoduli[g,n](\mu)^{\even}$ satisfying 
	\[h^0\Bigl(X, \sum_{i=1}^{n}\alpha_iz_i\Bigr) \geq 2\] 
	is a divisor. 
\end{prop}
\begin{proof}
	The generalized Weierstrass divisor $W_\mu(\alpha)$ is the locus where the morphism of vector bundles
	\[\phi\colon\mathcal{H}/\OO(-1) \rightarrow \mathcal{F}_{\alpha}\]
	degenerates (see Proposition \ref{prop:Wclass}). It follows that $W_\mu(\alpha)$ has codimension at most one in $\omoduli[g,n](\mu)^{\even}$. Hence, we only need to show that it is not the whole stratum. 
	\par
	To this end, we consider a partial closure of the generalized Weierstrass divisor in the locus of curves with rational tails, where rational tails arise when some marked points meet. For notation simplicity, we still denote this partial closure by $W_\mu(\alpha)$. We can always find suitable marked points to merge together, making it sufficient to prove the result for a stratum with $n-1$ marked points. More explicitly, if $n\geq 3$, it is impossible to have $\alpha_i + \alpha_j = \frac{m_i}{2} + \frac{m_j}{2}$ for all $1\leq i < j\leq n$. Consequently, there are two indices~$i \neq j$ such that $\alpha_i + \alpha_j \neq \frac{m_i}{2} + \frac{m_j}{2}$. Let 
	\[ \mu' \= (m_1, \ldots, \hat{m}_i, \ldots, \hat{m}_j, \ldots, m_n, m_i+m_j)\,.\]
	Consider the boundary divisor $D_\Delta$ for $\Delta\in \LG_1$, consisting of curves 
	\[ (X_0, z_1, \ldots, \hat{z}_i, \ldots, \hat{z}_j, \ldots, z_n, q, \omega_0) \in \omoduli[g,n-1]^{\even}(\mu') \]
	on the top level, attached via the point $q$ to a rational tail $(R, z_i, z_j, q, \omega_1)$ in the stratum $\omoduli[0,3](m_i,m_j, -m_i-m_j-2)$. If we can show that the generalized Weierstrass ``divisor'' is not equal to the whole stratum for the partitions $\mu'$ and $\alpha' = (\alpha_1, \ldots, \hat{\alpha}_i, \ldots, \hat{\alpha}_j, \ldots, \alpha_n, \alpha_i+\alpha_j)$, then $W_\mu(\alpha)$ does not contain all of~$D_\Delta$, and hence does not equal $\omoduli[g,n](\mu)^{\even}$. By repeating this procedure, we can assume that $n = 2$ from now on.
\par
We consider a stratum $\omoduli[g,2](m, 2g-2-m)^{\even}$ with $m\leq g-1$, and a partition $(\alpha, g-1-\alpha)$ satisfying $\alpha \neq \frac{m}{2}$ and $0 \leq \alpha \leq m$. Our goal is to show that 
\be \label{eq:h0isone}
h^0\Bigl(X, \alpha z_1 + (g-1-\alpha)z_2\Bigr) = 1
\ee
for a generic element $(X,z_1,z_2, \omega) \in \omoduli[g,2](m, 2g-2-m)^{\even}$.
\par 
To see this,  suppose $h^0(X, \alpha z_1 + (g-1-\alpha)z_2) \geq 2$ generically, and we will derive a contradiction. Consider the chain of divisors 
\[ \textbf{D}_\bullet(z_1,z_2) \= 0< z_1< 2z_1<\cdots < \alpha z_1 < \alpha z_1+z_2 <\cdots < \alpha z_1 + (g-1-\alpha)z_2\,. \] 
The hypothesis $h^0(X, \alpha z_1 + (g-1-\alpha)z_2) \geq 2$ implies the existence of a $g^1_{g-1}$ with vanishing orders $(0,g-1)$ with respect to $\textbf{D}_\bullet(z_1,z_2)$. Next, consider the clutching map 
\bes
\bP\omoduli[g-1,1](2g-4)^{\odd} \times
\bP\omoduli[1,3](2-2g, m, 2g-2-m)^{\odd} \rightarrow \bP\Xi \overline{\cM}_{g,2}(m,2g-2-m)^{\even}
\ees
which identifies the zero of order $2g-4$ with the pole of order $2-2g$ on the left-hand side to form a node on the right-hand side. By the preceding hypothesis, the 
underlying curve $(X_0\cup_qE,z_1,z_2)$ of a generic element in the image admits a limit $g^1_{g-1}$ with ramification orders $\geq (0, g-1)$ with respect to $\textbf{D}_\bullet(z_1,z_2)$. Lemma \ref{lemma:two ramification conditions} then implies that the $E$-aspect has vanishing orders $\leq (0,g-1)$ at the node~$q$. 
If the vanishing orders at $q$ are $(0,g-1)$, then using Lemma \ref{lemma:two ramification conditions} again, we obtain that 
\[ (g-1)q\sim \alpha z_1 + (g-1-\alpha)z_2 \,\sim\, \frac{m}{2}z_1 + (g-1-\frac{m}{2})z_2\,. \]
In particular, $z_1-z_2$ has to be torsion. However, this is not true for a generic element $(E, q, z_1, z_2, \omega_E)$ in $\omoduli[1,3](2-2g, m, 2g-2-m)^{\textrm{odd}}$. Hence, the vanishing orders of the $E$-aspect at $q$ are less than or equal to $ (0,g-2)$. Using the definition of limit linear series (see Definition \ref{def:limit linear series}), this implies that the vanishing orders of the $X_0$-aspect at $q$ are greater than or equal to $(1,g-1)$. In particular, a generic element $(X_0,q,\omega_0)$ in the stratum $\omoduli[g-1,1](2g-4)^{\textrm{odd}}$ admits a $g^1_{g-1}$ with vanishing orders greater than or equal to $(1,g-1)$ at $q$. This condition can be rewritten as $h^0(X_0, (g-2)q) \geq 2$, which is false for a generic $(X_0,q,\omega_0)$ in $\omoduli[g-1,1](2g-4)^{\textrm{odd}}$. Hence, our hypothesis was incorrect, and generically, \eqref{eq:h0isone} holds as desired.
\end{proof}
\par
As a consequence of this proposition and the preceding discussion, we obtain:
\par
\begin{cor} \label{cor: same formula}
Let $\mu = (m_1,\ldots, m_n)$ be an even partition of $2g-2$ and $\alpha = (\alpha_1, \ldots, \alpha_n)$ be a partition of $g-1$ satisfying $0\leq \alpha_i \leq m_i$ for all $i$ and $\alpha \neq \frac{\mu}{2}$. Then the statements of Proposition~\ref{prop:Wclass} and Lemma~\ref{le:optdifference} hold for the divisor
\be 
\wt{W}_\mu(\alpha) \= \{ \mathrm{rank}(\phi) < g-1\} \subset \bP\LMS[\mu][g,n]^{\even}.
\ee 
\end{cor}

\subsection{Weierstrass divisors with extra vanishing}
\label{sec:VO}

Let $\mu = (m_1, \ldots, m_n)$ be a partition of $2g-2$, where all its entries are even, and let $\alpha_i = \frac{m_i}{2}$ for every $1\leq i \leq n$. For the even spin component of the stratum $\bP\omoduli[g,n](\mu)$, the locus $W_{\mu}(\alpha)$, defined in~\eqref{eq:initdefGenW}, contains the entire even spin component by definition. 
\par 
To obtain divisorial conditions on $\bP\omoduli[g,n]^{\even}(\mu)$, we instead consider the partition $\alpha + e_j$,  obtained by adding $1$ to the $j$-th entry of $\alpha$ and leaving the other entries unchanged. In this situation, we consider the \emph{Weierstrass divisors with extra vanishing}: 
\[
W_{\mu}^{+j} ({\mu}/{2}) \= \Bigl\{ (X, \bfz, \omega) \in \bP\omoduli[g,n]^{\even}(\mu)\,:\,
h^0\Bigl(X, z_j + \sum_{i=1}^n \frac{m_i}{2} z_i\Bigr) \geq 3\Bigr \}\,.
\]
In particular, if we restrict to the partition $\mu = (2g-2)$, 
the Weierstrass divisor with extra vanishing is simply
\[ W^+ \coloneqq W_{(2g-2)}^{+1}(g-1)
\= \Bigl\{ (X, z, \omega) \in \bP\omoduli[g,1]^{\even}(2g-2)\,:\,
h^0\bigl(X, g z\bigr) \geq 3\Bigr \}\,.\]
\par
Our approach to computing the class of~$W^+$ is somewhat indirect: we pull back
twisted generalized Weierstrass divisors in genus~$g+1$ via the clutching map~$\zeta_E$. The resulting divisor might not be equal to~$W^+$ (see Corollary~\ref{cor: obtaining Weierstrass divisor}), but it will suffice for our application. For this purpose, we need to determine the exact vanishing order of the morphism~\eqref{eq:defgenWG} along the divisor $D_{\Gamma_1}$.
\par
We perform this computation in a general context, in genus~$g+1$, but we
need to impose a restriction on the partition~$\alpha$. Recall the definition of a good chain of divisors in~Definition~\ref{def:good-D}. 
\par
\begin{definition} Let $\mu = (m_1,\ldots,m_n)$ be a positive partition of~$2g$, and $\alpha = (\alpha_1, \ldots, \alpha_n)$ a partition of~$g$. For an index $1\leq k\leq n$, the partition $\alpha$ is called \emph{$k$-saturated (with
respect to $\mu$)} if $\alpha_k = m_k$ and $\alpha_j \leq m_j$ for each $j\neq k$. Moreover, the partition $\alpha$ is called \emph{good} if there exists a good chain of divisors $\textbf{D}_\bullet(\textbf{z})$ satisfying $D_g = \sum_{i=1}^{n} \alpha_i z_i$. 
\end{definition}
\par
\begin{prop} \label{prop: exact multiplicity} Let $\mu$ be a positive partition
of~$2g$ and let $\alpha = (\alpha_1, \ldots, \alpha_n)$ be a good $k$-saturated partition of $g$. If $\mu$ has at least one odd entry, then the morphism 
\[ \mathcal{H}/\OO(-1) \rightarrow \mathcal{F}_{\alpha} \]
on $\bP\LMS[m_1,\ldots, m_n][g+1,n]$ vanishes along~$D_{\Gamma_1}$ with multiplicity exactly
\[ 1 + 2 + 3 +\cdots + (g-2) + g \= \frac{g (g-1)}{2} + 1.  \]
Similarly, if $\mu$ is an even partition, the morphism $\mathcal{H}/\OO(-1) \rightarrow \mathcal{F}_{\alpha}$ on the stratum $\bP\LMS[m_1,\ldots, m_n][g+1,n]^{\even}$ vanishes along $D_{\Gamma_1}$ with multiplicity exactly $\frac{g (g-1)}{2} + 1.$
\end{prop}
\par
This is exactly one more than the corresponding coefficient in the estimate of $\Delta_\cL \wt{W}_\Gamma$ from
Lemma~\ref{le:optdifference}, where $2\alpha^\bot = m^\bot$ since all legs are
at the bottom in~$\Gamma_1$, and where the second summand equals $g(g-1)/2$.
\par
\begin{proof}
We consider a good $k$-saturated chain of divisors satisfying 
\[ D_g \= \sum_{i=1}^{n} \alpha_i z_i\,. \] 
Let $\pi \colon \mathcal{X} \rightarrow \Delta$ be a $1$-dimensional family of curves
in $\omoduli[g+1](\mu)$, whose central fiber is a generic element of $D_{\Gamma_1}$,  
denoted by $(X_0\cup_q E, z_1,\ldots, z_n)$. Let $\sigma_0, \sigma_1, \sigma_2, \ldots, \sigma_{g-1}$ be sections of $\omega_\pi((2g-1)X_0)$ such that their restrictions to~$E$, denoted by $\sigma^E_0, \sigma^E_1, \sigma^E_2, \ldots, \sigma^E_{g-1}$, have vanishing orders 
	\[(g-1,g+1,g+2,\ldots, 2g-2,2g)\] 
	at the point $q$. Multiplying the sections $\sigma_0, \ldots, \sigma_{g-1}$ by suitable powers of the parameter $t$ on $\Delta$ we obtain sections of $\omega_\pi(E)$. We denote by $\sigma^{X_0}_0, \sigma^{X_0}_1, \sigma^{X_0}_2, \ldots, \sigma^{X_0}_{g-1}$ the restrictions of these sections to $X_0$. For a section $\sigma_i$, we denote by $a_i$ and $b_i$ the vanishing orders of $\sigma_i^E$ and $\sigma_i^{X_0}$ at $q$. The result of \cite[Proposition 2.2]{limitlinearbasic} and its proof imply that $a_i + b_i \geq 2g$, and, in the case of equality, the section $\sigma_i$ vanishes along $X_0$ with multiplicity $a_i$. We know from Proposition \ref{prop:even-general-vanishing} that the limit canonical series is refined. In particular, the sections $\sigma_0, \sigma_1, \sigma_2, \ldots, \sigma_{g-1}$ vanish along $X_0$ with multiplicities exactly $(g-1,g+1,g+2,\ldots, 2g-2,2g)$. Therefore, the sections $t^{g}\sigma_0, t^{g-2}\sigma_1, \ldots, t\sigma_{g-2}$, and $\sigma_{g-1}$ can be viewed as sections of $\omega_\pi$. 
\par 	
From the definition of a good chain of divisors, we conclude that
	\[ \langle\sigma^E_0, \sigma^E_1, \sigma^E_2, \ldots, \sigma^E_{g-1}\rangle\rightarrow \mathcal{F}_{\alpha}\]
	is an isomorphism. Hence, it follows that the morphism $\mathcal{H}/\OO(-1) \rightarrow \mathcal{F}_{\alpha}$
	vanishes with multiplicity $1 + 2 + 3+\cdots + (g-2) +g = \frac{g (g-1)}{2}+1$ along $D_{\Gamma_1}$.
\end{proof}



Proposition \ref{prop: exact multiplicity} gives us the exact vanishing multiplicity of the morphism $$\mathcal{H}/\OO(-1) \rightarrow \mathcal{F}_{\alpha}$$ along $D_{\Gamma_1}$. In particular, the pullback of the generalized Weierstrass divisor $W_\mu(\alpha)$ via the map 
\bes
{\zeta_E}\colon \bP\LMS[2g-2][g,1]^{\even}  \rightarrow \bP\LMS[\mu][g+1,n] \,
\ees
gives an effective divisor on $\bP\LMS[2g-2][g,1]^{\even}$ and its class is readily computable. 

However, we would like a concrete geometric description of this pullback. In order to do this, we need to specialize to an even partition $ \mu = (2,2m_2,\ldots, 2m_n,0)$
and look for a good saturated partition of $g$ in this situation.  For this purpose, we consider two chains of divisors $\textbf{D}^1_\bullet(\textbf{z})$ and $\textbf{D}^2_\bullet(\textbf{z})$ given as 
\[\textbf{D}^1_\bullet(\textbf{z}) = (0 = D^1_0 < D^1_1 < D^1_2<\cdots < D^1_{2g})\]
and 
\[\textbf{D}^2_\bullet(\textbf{z}) = (0 = D^2_0 < D^2_1 < D^2_2<\cdots < D^2_{2g})\]
satisfying 
\begin{itemize}
	\item $D^1_i = D^2_i$ for $0\leq i \leq g-1$; 
	\item $D^1_{g-1} = D^2_{g-1} = z_1 + m_2 z_2 + \cdots + m_{n-1} z_{n-1} + (m_n-1)  z_n$; 
	\item $D^1_g = z_1 + \sum_{i=2}^{n} m_i z_i$ and  $D^2_g = 2z_1 + \sum_{i=2}^{n-1} m_i z_i + (m_n-1)z_n$; 
	\item $D^1_{g+j} = z_1 + \sum_{i=2}^{n} m_i z_i + j\cdot p$ and $D^2_{g+j} = (2+j)z_1 + \sum_{i=2}^{n-1} m_i z_i + (m_n-1)z_n$ for every $1\leq j \leq g$. 
\end{itemize}

Let $(X_0\cup_qE, \textbf{z}, \omega)$ be a generic element in the boundary divisor $D_{\Gamma_1}$ of the stratum $\bP\LMS[m_1,\ldots, m_n][g+1,n]^{\even}$. Proposition \ref{prop:even-general-vanishing} and Proposition \ref{prop: even cases} imply that $\textbf{D}^2_\bullet(\textbf{z})$ is a good chain of divisors. In particular, $\alpha = (2, m_2,\ldots, m_{n-1}, m_n-1,0)$ is a good $1$-saturated partition of $g$. Proposition \ref{prop: exact multiplicity} implies that the morphism 
 \[ \mathcal{H}/\OO(-1) \rightarrow \mathcal{F}_{\alpha} \]
 vanishes along $D_{\Gamma_1}$ with multiplicity exactly $\frac{g(g-1)}{2} +1$. In this setting, Proposition \ref{prop: even cases} and Corollary \ref{cor: obtaining Weierstrass divisor} imply that: 
 \begin{cor} \label{cor: pullback-Weierstrass} Let $ \mu = (2,2m_2,\ldots, 2m_n,0)$ and  $\alpha = (2, m_2,\ldots, m_{n-1}, m_n-1,0)$ as above. We have the inclusion
 \[ \xi_E^{-1}(\overline{W}_\mu(\alpha)) \subseteq W^+ \cup \ \textrm{boundary divisors}.\]
 \end{cor} 
 
 The following theorem is a consequence of the discussion above and the results in Subsection \ref{subsec:Weierstrass-nonminimal}. The methods of \cite{Kodstrata} can be used to compute the class of the generalized Weierstrass divisor $\overline{W}_\mu(\alpha)$ on even spin components (see Subsection \ref{sec:tWP} and Corollary \ref{cor: same formula}).

\begin{theorem} \label{theorem:class_twisted_Weierstrass}
	The class 
\ba \ 
[W^+] &\,\coloneqq\, \Bigl(\frac{g(g-1)}{2} +1\Bigr)\psi-\lambda_1 + \xi-
\sum_{\Delta\in {\LG_1}} \Bigl(\frac{P-P_{-1}}{8} + \frac{v^\top-1}2\Bigr)\ell_\Delta[D_\Delta] \\ 
&\= \frac{g+11}{2g-2}\lambda_1 -\frac{g+3}{8g-8}[D_h] \\
&\phantom{\=} \quad
- \sum_{\Delta\in {\LG_1}} \Bigr(\frac{\kappa^\bot}{\kappa_{(2g-2)}}
\bigl(1 + \frac1{2g-1}\bigr) - \frac{1}{2g-1}  + \frac{v^\top-1}2\Bigr)
\ell_\Delta [D_\Delta]\\
&\,=:\, w_\lambda \lambda_1 - w_{\hor} [D_h] - \sum_{\Delta\in {\LG_1}} w_\Gamma
\ell_\Delta [D_\Delta]
\ea
is the class of an effective divisor in $\bP\LMS[2g-2][g,1]^{\even}$, which is set-theoretically supported on the boundary divisors and on the Weierstrass divisor with extra
vanishing~$W^+$.
\end{theorem}
\par
We remark that this class differs from the expression of the twisted generalized
Weierstrass divisor in the odd spin case by adding~$\psi$ with coefficient one.
Note that
\be
\kappa_{(2g-2)} \= \frac{4g(g-1)}{2g-1}, \qquad \kappa^\bot \= \kappa_{(2g-2)}
- (P - P_{-1})\,.
\ee 
\par
\begin{proof} Let $ \mu = (2,2m_2,\ldots, 2m_n,0)$ and  $\alpha = (2, m_2,\ldots, m_{n-1}, m_n-1,0)$.  Let $\delta^{\min}_\cL$ denote the right-hand side
of~\eqref{eq:Deltaestimate}.  As a consequence of Corollary~\ref{cor: same formula} and
Proposition~\ref{prop: exact multiplicity}, the class 
\ba \label{eq:class in g+1}
\sum_{i=1}^{n}\frac{\alpha_i(\alpha_i+1)}{2}\psi_i - \lambda_1 +\xi &- \Bigl(\frac{(g-1)g}{2}+1\Bigr)D_{\Gamma_1} - \sum_{\Gamma\in \LG_1, \Gamma\neq \Gamma_1}
\delta^{\min}_\cL \cdot [D_\Gamma]\,
\ea
represents an effective divisor in $\bP\LMS[\mu][g+1,n]^{\even}$ which does not contain $D_{\Gamma_1}$
in its support. The pullback of this class via the map~$\zeta_E$ 
can be computed using Proposition~\ref{prop: pullback}. The expression simplifies
because all boundary divisors of $\bP\LMS[\mu][g+1,n]$ that pull back non-trivially
have all their marked points on level~$-1$, leading to $\alpha^\bot =\frac{1}{2}m^\bot = g$
for all divisors with non-trivial pullback. Using that
$ (P-P_{-1})/8 = \vartheta_{\mu, \alpha}{\kappa_{\mu_\Gamma^{\bot}}}/ {\kappa_\mu}
- \vartheta_{\mu_\Gamma^{\bot}, \alpha_\Gamma^{\bot}}$
yields the expression claimed in the theorem.

Finally, Corollary \ref{cor: pullback-Weierstrass} gives an explicit geometric description of the pullback as the union of $W^+$ and boundary divisors. 
\end{proof}

\par

Next, we consider Weierstrass divisors with extra vanishing when the even partition $\mu$ is not minimal (i.e., $\mu \neq (2g-2)$). We will show that these loci are indeed divisors and that they have distinct supports. 

\begin{prop} \label{prop: properdistinctdivisors} Let $\mu = (m_1,\ldots, m_n)$ be a positive partition of $2g-2$, where all its entries are even, and let $\alpha = (\alpha_1, \ldots, \alpha_n)$ be the partition given by $\alpha_i = \frac{m_i}{2}$ for all $1\leq i \leq n$. Then, for every $1 \leq j \leq n$, the locus
	\[
	W_{\mu}^{+j} \= \Bigl\{ (X, \bfz, \omega) \in \bP\omoduli[g,n]^{\even}(\mu)\,:\,
	h^0\Bigl(X, z_j + \sum_{i=1}^n \alpha_i z_i\Bigr) \geq 3\Bigr \}\,
	\]
	has a divisorial component. Moreover, these $n$ divisors obtained as $j$ varies from $1$ to $n$ have distinct supports.  
\end{prop}
\par
\begin{proof}
     Consider $(X_0, q, \omega_0)$ generic in $\omoduli[g-1,1](2g-4)^{\even}$ and let 
	\[ (E,q,z_1,\ldots, z_n, \omega_1) \in \omoduli[1,n+1](2-2g, m_1,\ldots, m_n)\]
	which satisfies the equivalence of divisors 
	\[(g-2)q \sim  -z_j + \sum_{i=1}^{n}\frac{m_i}{2}z_i\,.\] 
	Gluing the two underlying curves yields an element in the boundary $D_{\Gamma_1}$ of $\bP\omoduli[g,n]^{\even}(\mu)$. Our goal is to show that $(X_0\cup_qE, \mathbf{z}, \omega)$ is contained in a divisorial component of $W_\mu^{+j}$. For this purpose, we will find an open subset $U\subseteq \bP\omoduli[g,n]^{\even}(\mu)$, containing $(X_0\cup_qE, \mathbf{z}, \omega)$, on which $W_\mu^{+j}$ is defined as a degeneracy locus of a morphism of vector bundles of rank $g-2$. 
	
	Consider the vector bundle morphisms $\varphi_\alpha\colon \mathcal{H} \rightarrow \mathcal{F}_\alpha$ and $\tau_j\colon \mathcal{F}_{\alpha} \rightarrow \mathcal{F}_{\alpha-e_j}$, respectively.  Since $h^0(X, \sum_{i=1}^n\alpha_iz_i) = 2$ generically on $ \bP\omoduli[g,n]^{\even}(\mu)$, there exists an open subset $U\subseteq \bP\omoduli[g,n]^{\even}(\mu)$ on which $\textrm{Im}(\varphi_\alpha)$ is a vector bundle of rank $g-2$. Moreover, for a limit canonical series on $(X_0\cup_qE, \mathbf{z}, \omega)$, the vanishing orders of the $E$-aspect at $q$ are greater than or equal to $(0,g-2, g,\ldots, 2g-4,2g-2)$. 
    
    We consider a chain of divisors 
      \[\textbf{D}^j(z_1, \ldots, z_n) \= \Bigl(0 = D_0 < D_1 < \cdots < D_{2g-2} \Bigr)\,\]
    satisfying $D_{g-2} = -z_j + \sum_{i=1}^{n} \alpha_iz_i$, $D_{g-1} = \sum_{i=1}^{n} \alpha_iz_i$, and $D_{2g-2} = \sum_{i=1}^{n} m_iz_i$. We will now use Lemma \ref{lemma:two ramification conditions} to describe the vanishing orders of the $E$-aspect with respect to the chain of divisors $\textbf{D}^j(z_1, \ldots, z_n)$. Note that there cannot be three vanishing orders greater than or equal to $g-1$, since there are only two vanishing orders at $q$ smaller than or equal to $g-1$. 
    
    Let $\mathcal{V}_E$ be the $E$-aspect of the limit canonical series. The map $\varphi_\alpha$ extends over $(X_0\cup_qE, \mathbf{z}, \omega)$ to a map $\mathcal{V}_E \rightarrow \mathcal{F}_\alpha$, and the image of this map has dimension $g-2$, since there are only two sections vanishing with order greater than or equal to $g-1$ at $\textbf{D}^j(z_1, \ldots, z_n)$. In particular, we can take the open $U$ to contain $(X_0\cup_qE, \mathbf{z}, \omega)$. 
    
    Next, we consider the morphism of vector bundles of rank $g-2$ on $U$ defined as the restriction of $\tau_j$ to $\textrm{Im}(\varphi_\alpha)$, i.e., $\overline{\tau}_j\colon \textrm{Im}(\varphi_\alpha) \rightarrow \mathcal{F}_{\alpha-e_j}$. In the interior of the stratum, the degeneracy locus of $\overline{\tau}_j$ is contained inside $W^{+j}_\mu$. On the boundary, we claim that $(X_0\cup_qE, \mathbf{z}, \omega)$ is contained in the degeneracy locus of $\overline{\tau}_j$. To see this, first note that the vanishing orders of the $E$-aspect at $q$ are greater than or equal to $(0,g-2, g,\ldots, 2g-4,2g-2)$. Lemma \ref{lemma:two ramification conditions} then implies that the vanishing orders of the $E$-aspect with respect to $\textbf{D}^j(z_1, \ldots, z_n)$ are smaller than or equal to $(0,1,2,\ldots, g-4, g-2,g-1, 2g-2)$. In fact, these are exactly the vanishing orders: 
    \begin{itemize}
    	\item The first $g-3$ entries are the lowest possible, and hence cannot be lower.
    	\item For $(X, \mathbf{z}, \omega)$ generic in $\bP\omoduli[g,n]^{\even}(\mu)$, the last two vanishing orders with respect to $\textbf{D}^j(z_1, \ldots, z_n)$ are greater than or equal to $g-1$ and $2g-2$, and hence the same is true for a limit canonical series in the boundary. 
    	\item The third biggest order is $g-2$ because of the equivalence $(2g-2)q \sim gq + D_{g-2}$. 	If the third biggest order were $g-3$, then Lemma \ref{lemma:two ramification conditions} would imply the existence of a section $\sigma$ vanishing with order $g$ at $q$ and order $g-3$ with respect to $\textbf{D}^j(z_1, \ldots, z_n)$. However, because $2gq \sim \div(\sigma)$, this section would vanish with order $g-2$ with respect to $\textbf{D}^j(z_1, \ldots, z_n)$, leading to a contradiction. 
    \end{itemize}  
 Since the last three vanishing orders with respect to $\textbf{D}^j(z_1, \ldots, z_n)$ are greater than or equal to $g-2$, it follows that $\overline{\tau}_j$
 degenerates at $(X_0\cup_qE, \mathbf{z}, \omega)$, as claimed. 
 
 For a generic element of the boundary divisor $D_{\Gamma_1}$, the vanishing orders with respect to $\textbf{D}^j(z_1, \ldots, z_n)$ are $(0,1,\ldots, g-3, g-1, 2g-2)$, and hence $\overline{\tau}_j$ does not degenerate along this divisor. On the other hand, the morphism $\overline{\tau}_j$ degenerates along the locus of curves $(X_0\cup_qE, \mathbf{z}, \omega)$ inside $D_{\Gamma_1}$ where the equivalence 	$(g-2)q \sim  -z_j + \sum_{i=1}^{n}\frac{m_i}{2}z_i$ is satisfied on the elliptic component $E$. Note that this locus has codimension $2$ in $ \bP\omoduli[g,n]^{\even}(\mu)$, while the locus where $\overline{\tau}_j$ degenerates has codimension $1$. Therefore, the former is contained in the boundary of a divisorial component of $W_\mu^{+j}$.
 
 Finally, we argue that  $(X_0\cup_qE, \mathbf{z}, \omega)$ is not contained in the boundary of any other Weierstrass divisor $W_\mu^{+k}$ for $k \neq j$. Indeed, the vanishing orders of the $E$-aspect with respect to $\textbf{D}^k(z_1, \ldots, z_n)$ are $(0,1,2,\ldots, g-3,g-1,2g-2)$. Because the third biggest vanishing order is smaller than $g-2$, it follows that the map $\overline{\tau}_k$ does not degenerate at $(X_0\cup_qE, \mathbf{z}, \omega)$, and hence it is not contained in $W^{+k}_\mu$. 
\end{proof}

\par 
Because generalized Weierstrass divisors share several similarities with the classical Weierstrass divisors, it is natural to ask whether they are also irreducible. It is an immediate consequence of Proposition \ref{prop: properdistinctdivisors} that this is not always the case. 
\par	
\begin{prop} \label{prop:genWisreducible}
  Let $n\geq 2$ and $\mu = (m_1,\ldots,m_n)$ a partition of $2g-2$ with all entries even. For any $1\leq j\neq k \leq n$, consider the partition $\alpha= (\alpha_1,\ldots, \alpha_n)$ of $g-1$ defined as 
	\[ \alpha \= \Bigl(\frac{m_1}{2}, \ldots, \frac{m_n}{2}\Bigr)-e_j + e_k\,,\]
where $e_j$ and $e_k$ are the standard unit vectors with all but one entry equal to zero. 
Then the generalized Weierstrass divisor 
\[
W_{\mu}(\alpha) \= \Bigl\{ (X, \bfz, \omega) \in \bP\omoduli[g,n](\mu)^{\even}\,:\,
h^0\Bigl(X, \sum_{i=1}^n \alpha_i z_i\Bigr) \geq 2\Bigr \}\, 
\]
is reducible.  
\end{prop}
\begin{proof}
We consider the Weierstrass divisors with extra vanishing: 
\[ 
W_{\mu}^{+j} \= \Bigl\{ (X, \bfz, \omega) \in \bP\omoduli[g,n](\mu)^{\even}\,:\,
h^0\Bigl(X, -z_j +\sum_{i=1}^n \frac{m_i}{2} z_i\Bigr) \geq 2\Bigr \}\,
\]
and 
\[ 
W_{\mu}^{+k} \= \Bigl\{ (X, \bfz, \omega) \in \bP\omoduli[g,n](\mu)^{\even}\,:\,
h^0\Bigl(X, -z_k  +\sum_{i=1}^n \frac{m_i}{2} z_i\Bigr) \geq 2\Bigr \}\,.
\]
Proposition~\ref{prop: properdistinctdivisors} shows that they have distinct supports. Moreover, by Serre duality, the condition defining $W_\mu(\alpha)$ is 
\[ h^0\Bigl(X, -z_k + z_j +\sum_{i=1}^n \frac{m_i}{2} z_i\Bigr) =  h^0\Bigl(X, -z_j + z_k +\sum_{i=1}^n \frac{m_i}{2} z_i\Bigr) \geq 2\,,\]
and hence $W_{\mu}^{+j} \subseteq W_\mu(\alpha)$ as well as $W_{\mu}^{+k} \subseteq W_\mu(\alpha)$. It follows that $W_\mu(\alpha)$ is reducible. 
\end{proof}
\par
To our knowledge, this is the first known example of reducible generalized Weierstrass divisors. 

%% file: sec_gentype.tex
\section{Certifying general type}
\label{sec:gentype}

To prove Theorem~\ref{intro:minimal}, we apply
Proposition~\ref{prop:GenTypeCrit}: There exists an effective linear combination of boundary divisors (denoted by $\cL \otimes \cO_{\bP\MScoarse}(-D)$ in
\cite[Proposition~3.2]{Kodstrata}) such that
\be \label{eq:ampleclass}
A \= \lambda + \epsilon (\cL \otimes \cO_{\bP\MScoarse}(-D))
\ee
is ample (see the proof of \cite[Theorem~1.1]{Kodstrata}). We will exhibit the
effective divisor~$E$ as a convex combination of the Weierstrass
divisor with extra vanishing and a Brill--Noether divisor. More precisely,
we follow the strategy of \cite[Section~8]{Kodstrata} and consider
the sum
\bas
& \frac{\kappa_{{(2g-2)}}}{2g} \Bigl( c_1\bigl(K_{\bP\MScoarse}\bigr)
- D_{\mathrm{NC}} \Bigr)
- y\frac{12}{w_\lambda}
[{W}^{+}] - (1-y)\cdot 2[{\BN}_{\mu}] \\
& \qquad \= s_{\hor}(y) [D_h] + \sum_{\Gamma \in \LG_1}  \ell_\Gamma
s_\Gamma(y) [D^{\rm H}_\Gamma] 
\eas
where
\ba
\label{eq:horstar}
s_{\hor}(y) &\= -1 -\frac{2g-2}{2g-1} +y
\frac{12 w_{\hor}}{w_\lambda}
+(2-2y)\frac{g+1}{g+3} \\ 
s_\Gamma(y) &\= c_\Gamma + y \frac{12 w_{\Gamma}}{w_{\lambda}}
+ (1-y)b_\Gamma\,
\ea
and the expressions of $w_{\hor}$, $w_\Gamma$, and $w_\lambda$ can be found in~\cite[Corollary 7.7]{Kodstrata}. 
With $b_{\NC}^\Gamma$ defined in~\eqref{eq:non-canonical-term}, we let
\ba
c_\Gamma &\= \frac{2g-2}{2g-1} \Big(N_\Gamma^{\bot} -
R_\Gamma \Big) -  \kappa_{\mu_{\Gamma}^{\bot}}\,,
\qquad
R_\Gamma =
\frac{b_{\mathrm{NC}}^\Gamma  + 1 +\delta_{\Gamma}^{\rm H}} {\ell_\Gamma}\,,
\\
b_\Gamma &\= \sum_{i=1}^{[g/2]}\sum_{e\in E(\Gamma)
	\atop e \mapsto \Delta_i} \frac{12i(g-i)}{(g+3) p_e} +  
\sum_{e\in E(\Gamma) \atop e \mapsto \Delta_{\irr}} \frac{2(g+1)}{(g+3) p_e}\,,  \\
\ea
and finally, $\delta_\Gamma^H$ is one if~$D_\Gamma$ has components
that are HBB, and zero otherwise. The goal is to find~$y \in [0,1]$ such
that~$s_{\hor}(y)$ and all the coefficients~$s_\Gamma(y)$ are strictly positive 
since we can then take~$\epsilon$ in~\eqref{eq:ampleclass} sufficiently small and
apply Proposition~\ref{prop:GenTypeCrit}. 
Inserting the above constants leads to the following result. 
\par
\begin{lemma} \label{le:shorLB}
The coefficient $s_{\hor}$ is positive if and only if
\bes
y \geq y_{\hor} := (7g + 77)/(2g^2 - 11g + 5)\,.
\ees
\end{lemma}
\par
In order to control the coefficients of the other boundary divisors, still following the strategy of \cite[Section~8]{Kodstrata}, recall from~\eqref{eq:defkappamu} and Theorem~\ref{theorem:class_twisted_Weierstrass} that
\bas
N^\top &\= 2g^\top - v^\top = P + v^\top, \qquad \kappa^\top = P - P_{-1}\,, \\
y \frac{12 w_{\Gamma}}{w_{\lambda}}
&\ = 12 y \Bigl( \ol{w}_\Gamma + \frac{(g-1)(v^\top-1)}{g+11} \Bigr)\,, \qquad
\ol{w}_\Gamma := \frac{2g-2  - P + P_{-1}}{(g+11)}\,, 
\eas
so that it suffices to control the positivity of
\begin{flalign} 
s_\Gamma(y)
&\,\geq\, \Big(\frac{12(g-1)}{g+11}y-\frac{2g-2}{2g-1}\Big)(v^\top-1)
+(1-y)b_\Gamma- P_{{-1}} -\frac{2g-2}{2g-1} R
\label{eq:gencGammaMainEstimate} \\	
&\,+\, \frac{1}{2g-1} P \,+\,12 y\cdot \ol{w}_{\Gamma}
-\frac{2g-2}{2g-1}   \qquad =: T_1 + T_2\,,
 \nonumber 
\end{flalign}
where $T_1$ and $T_2$ refer to the terms in the first and second lines, respectively.  
\par
\begin{proof}[Proof of Theorem~\ref{intro:minimal}]
We start with the strategy for large~$g$. First, we need that
\be \label{eq:vtoppos}
y \geq (g + 11)/(12g - 6)
\ee
so that the coefficient of $v^\top -1$ in $T_1$ is positive. Second, if
\ba \label{eq:T1pos}
y \,\leq\, \frac{g - 5}{4g - 4}  & \qquad \text{if $g$ is odd}; \\
y \,\leq\, \frac{g^2 - 7g}{4g^2 + 16g - 8}  & \qquad \text{if $g$ is even}, \\
\ea
then the remaining term of $T_1$ is positive by the proof of
\cite[Lemma~8.6]{Kodstrata} (see the NCT-term; the lemma gives the large~$g$ limit and use the fact that there are no RBT graphs for the minimal strata), except if~$\Gamma$ is a HBB graph, in which case there is an additional term of $-1/\ell_\Gamma$. Since in a minimal stratum either the bottom-level vertex has positive genus, or there are at least two top-level vertices, or there are at least two edges from the unique top-level vertex, we know that $P \leq 2g-3$. This implies that $T_2 \geq P_{-1} \geq 1/\ell_\Gamma$ if
\be \label{eq:T2pos}
((24P_{-1} + 24)g -12P_{-1} - 1))y  - g - 11 \geq P_{-1}\,.
\ee
This is a lower bound on~$y$, which decreases as a function of $P_{-1}$. The critical case is thus $P_{-1} = 1$, which gives the lower bound
\be \label{eq:T2posMax}
y \geq (g + 12)/(48g-24)\,.
\ee
This implies~\eqref{eq:T2pos}. We need to control the lower bounds given by
$y_{\hor}$, by~\eqref{eq:vtoppos}, and by~\eqref{eq:T2posMax} against the upper bound
given by~\eqref{eq:T1pos}. This coarse strategy can only work while the lower bound in Lemma~\ref{le:shorLB} does not contradict the upper bound in~\eqref{eq:T1pos}, i.e., for $g \geq 29$. Indeed, taking $y = y_{\hor} + 10^{-5}$ works for
odd $g$ in the range $31 \leq g \leq 45$ and for
even $g$ in the range $34 \leq g \leq 46$, while taking $y= 0.15$ works for all $g \geq 47$.
\end{proof}